\documentclass[12pt, a4article]{amsart}   	
\usepackage[top=1.15in, bottom=1.15in, left=1.17in, right=1.17in]{geometry}

\usepackage[colorlinks=true, pdfstartview=FitV, linkcolor=blue, citecolor=blue, urlcolor=blue, breaklinks=true]{hyperref}
\usepackage{amssymb,amscd,graphics, amsfonts, multicol}
\usepackage{graphics}
\usepackage{stmaryrd}
\usepackage{DotArrow}
\usepackage{amsmath, amsthm,wasysym}
\usepackage{epsf}
\usepackage{xypic}
\numberwithin{equation}{section}
\usepackage{tikz-cd}
\usepackage{color}
\usepackage[normalem]{ulem}
\theoremstyle{plain}

\newtheorem{theorem}{Theorem}[section]
\newtheorem{corollary}[theorem]{Corollary}
\newtheorem{proposition}[theorem]{Proposition}
\newtheorem{lemma}[theorem]{Lemma}

\theoremstyle{definition}

\newtheorem{remark}[theorem]{Remark}

\newtheorem{example}[theorem]{Example}
\newtheorem{definition}[theorem]{Definition}
\definecolor{fondo}{rgb}{0.898,0.996,0.898}
\newtheorem*{claim}{Claim}

\newcommand{\R}{{\mathbb R}}

\newcommand{\bfa}{{\bf a}}
\newcommand{\bfb}{{\bf b}}

\newcommand{\bfSigma}{{\bf \Sigma}}
\newcommand{\bfDelta}{{\bf \Delta}}

\newcommand{\mcC}{{\mathcal C}}
\newcommand{\mcD}{{\mathcal D}}
\newcommand{\mcF}{{\mathcal F}}
\newcommand{\mcH}{{\mathcal H}}
\newcommand{\mcP}{{\mathcal P}}
\newcommand{\mcQ}{{\mathcal Q}}

\newcommand{\mfS}{{\mathfrak S}}
\newcommand{\ue}{{\underline e}}
\newcommand{\uu}{{\underline u}}
\newcommand{\uv}{{\underline v}}
\newcommand{\uw}{{\underline w}}
\newcommand{\ux}{{\underline x}}
\newcommand{\uy}{{\underline y}}

\newcommand{\MNaR}{M_{N,r}(\R)}

\DeclareMathOperator{\codim}{codim}

\newcommand{\bc}{\mathbf{c}}

\newcommand{\ba}{\mathbf{a}}

\newcommand{\bb}{\mathbf{b}}

\title{Higher rank Gelfand-Kapranov-Zelevinsky fans}
\author{Rocco Chiriv\`i}
\address{Dipartimento di Matematica e Fisica ``Ennio De Giorgi'', Universit\`a del Salento, Lecce, Italy}
\email{rocco.chirivi@unisalento.it}

\author{Martina Costa Cesari}
\address{European Research Council Executive Agency, Brussels, Belgium}
\email{martina.costa.cesari@live.it}

\author{Xin Fang}
\address{Lehrstuhl f\"ur Algebra und Darstellungstheorie, RWTH Aachen, Pontdriesch 10-16, 52062 Aachen, Germany}
\email{xinfang.math@gmail.com}

\author{Peter Littelmann}
\address{Department Mathematik/Informatik, Universit\"at zu K\"oln, 50931, Cologne, Germany}
\email{peter.littelmann@math.uni-koeln.de}

\begin{document}
\maketitle
\begin{abstract}
We define and study the higher rank GKZ-fans of point configurations, serving as the set of discrete and homogeneous quasi-valuations on the homogeneous coordinate ring of the associated toric variety, where the rank one cases coincide with the usual GKZ-fans. Such a quasi-valuation is then used to degenerate the toric variety flatly to a reduced union of toric varieties, which encodes the polytopal subdivision arising from the point in the higher rank GKZ-fan.
\end{abstract}

\section{Introduction}

In the paper we fix $\mathbb{K}$ to be an algebraically closed field of characteristic zero.

Given an embedded toric variety $X\hookrightarrow \mathbb{P}(V)$ which is not necessarily normal, we want to attach to the variety a set of discrete and
homogeneous quasi-valuations defined on the homogeneous coordinate ring $\mathbb{K}[X]$ with value in $\mathbb{Q}^N\cup\{\infty\}$. 
For $N=1$, the classical GKZ-cones \cite{GKZ} serve here as a model to endow this set with a reasonable polyhedral structure, but for general $N>1$, this classical model needs a lift to higher rank maps and new polyhedral geometric structures.

\subsection{GKZ-cones and GKZ-fans}
The GKZ-cones and GKZ-fans, named after Gelfand, Kapranov and Zelevinsky \cite{GKZ}, originate in their work on $A$-hypergeometric systems. Their main motivation was to give a modern treatment to the classical resultants and discriminants using Newton polytopes and toric geometry.

Precisely, let $A\subseteq \mathbb{Z}^{n-1}$ be a finite set of points, called a point configuration. Their convex hull $P$ is a polytope in $\mathbb R^{n-1}$ with integral vertices. A point in a GKZ-cone is a coherent assignment of a real number to each element in $A$, called a height function, \emph{i.e.}, a function $A\to \mathbb{R}$. Geometrically, one looks at the graph of this height function, takes the convex hull of the graph, and projects down all ``upper faces'' of this convex hull to the polytope $P$. The images of these projected faces form a polytopal subdivision of $P$ into polytopes. All polytopal subdivisions of $P$ arising in this way from a height function are called \emph{regular} (in \cite{GKZ} they were called \emph{coherent}). A GKZ-cone gathers all height functions giving the same polytopal subdivision of $P$; in \emph{loc.cit.} it has been shown that all GKZ-cones can be glued together into a complete fan, called the \emph{secondary fan}.

The algebraic geometry behind a point configuration is the toric variety $X_A$, defined as the closure of the $(\mathbb{K}^*)^A$-orbit though all characters in $A$. Such a toric variety carries naturally an embedding into the projective space of the characters. Let $I_A$ be the defining ideal of $X_A$ in this embedding. Sturmfels \cite{StuToric} studied the Gr\"obner bases of the ideal $I_A$, and made use of these height functions to construct Gr\"obner degenerations of the embedded toric variety $X_A$ into  unions of toric schemes. More precisely, using the height function, he constructed a flat family $\mathcal{X}\to\mathbb{A}^1$ whose generic fiber is isomorphic to $X_A$ and the reduced structure of the special fiber encodes the combinatorics of the polytopal subdivision.

\subsection{Higher rank GKZ-fans}

In \cite{GKZ}, all possible such height functions which induce regular subdivisions are studied: they form a complete polyhedral fan in $\mathbb{R}^A$, the vector space of functions from $A$ to $\mathbb{R}$. Motivated by the higher rank quasi-valuations constructed in \cite{CCFL}, we intend to associate to each point in $A$ a vector in $\mathbb{R}^N$ for some $N\geq 1$ and not just a real number. This natural idea leads to the following questions:

\begin{enumerate}
\item[(1)] How to define analogues of the GKZ-cones in this setup? What is the discrete geometric structure of the analogues of the GKZ-cones?
\item[(2)] With the structure in (1), can these analogues of GKZ-cones be glued together into a complete fan?
\item[(3)] How to define a quasi-valuation from a point in the analogue of a GKZ-cone? Will similar results in \cite{CCFL} still hold for this quasi-valuation?
\end{enumerate}

The goal of the paper is to give affirmative answers to all three questions.

Fix two integers $n\geq 2$ and $N\geq 1$. Indexing elements in $A$ by $\chi_1,\ldots,\chi_r\in\mathbb{R}^{n-1}$, a map from $A$ to $\mathbb{R}^N$ can be represented by a matrix in $M_{N,r}(\mathbb{R})$ where the $j$-th column is the image of $\chi_j$ under the map. Instead of the pairing between $\mathbb{R}^r$ and itself via the Euclidean inner product $\mathbb{R}^r\times\mathbb{R}^r\to\mathbb{R}$ in the polyhedral geometry, what we have in our setup is only a bilinear map $\mu: M_{N,r}(\mathbb{R})\times\mathbb{R}^r\to\mathbb{R}^N$ given by matrix multiplication. To speak about polyhedral objects in polyhedral geometry, one uses essentially the standard order structure on $\mathbb{R}$, the target space of the inner product. When the image lies in a vector space $\mathbb{R}^N$ but $\mathbb{R}$, to define an analogue of a polyhedral object, a total order has to be fixed on $\mathbb{R}^N$. In this paper we fix on $\mathbb{R}^N$ the lexicographic order, and define similarly the closed half spaces for $\Phi\in M_{N,r}(\mathbb{R})$ and $\underline{w}\in\mathbb{R}^r$:
$$\mcH_{ {\Phi}}^\leq=\{\uv\in \mathbb R^r\mid \Phi\cdot\uv\le \mathbf 0\}\quad\text{and}\quad
H_{ {\underline{w}}}^\leq =\{\Psi\in M_{N,r}(\mathbb R)\mid \Psi\cdot\underline{w}\le \mathbf 0\}.$$

A $\mu$-polyhedral cone in $\mathbb{R}^r$ or in $M_{N,r}(\mathbb{R})$ is a finite intersection of the corresponding closed half spaces. Polyhedral cones in $\mathbb{R}^r$ are $\mu$-polyhedral. What makes the difference between these two notions is the lexicographic order: in general, the closed half spaces are no longer closed in the Euclidean topology on $\mathbb{R}^r$ and on $M_{N,r}(\mathbb{R})$. To get rid of this issue, we work with the order topology on $\mathbb{R}^r$ and the product order topology on $M_{N,r}(\mathbb{R})$. When $N\geq 2$, these topologies are strictly finer than the Euclidean topology and its products. With this topology, the real vector space $\mathbb{R}^N$ is no longer a topological vector space, and moreover, the famous Farkas lemma on the existence of separating hyperplanes fails: studying geometry of $\mu$-polyhedral cones needs new tools.

Parallel to the usual polyhedral geometry, in Section \ref{sec:polyhedral:cones:and:polars} we develop the discrete geometry of $\mu$-polyhedral cones such as polar duality (Theorem \ref{thm:polar:copolar}), faces and co-faces (Proposition \ref{prp:coface-cofan} and Proposition \ref{prp:face:coface}), fans and co-fans (Theorem \ref{thm:fan:cofan}) and their basic properties and behaviors under polar duality. The new tool here is the notion of co-faces of a convex cone, which is new up to our knowledge. It enables us to give short proofs for properties of faces in the $\mu$-polyhedral setup.

The higher rank GKZ-cones are defined and studied in Section \ref{sec:HiDiGKZ}. Let $P$ be the convex hull of $A$ and $\mathcal{Q}$ be a polytopal subdivision of $P$. The higher rank GKZ-cone $C(\mathcal{Q},N)$ is defined similar to the one in \cite{GKZ} by replacing functions to $\mathbb{R}$ by functions to $\mathbb{R}^N$, and the total order on $\mathbb{R}$ by the lexicographic order on $\mathbb{R}^N$ (Definition \ref{def:gkz:open:cone}, and Definition \ref{def:gkz:close:cone} for its closure). To show that the higher rank GKZ-cones are $\mu$-polyhedral, we introduce the GKZ-condition-cones (Definition \ref{def:condition:cone}) and show that the higher rank GKZ-cones are polar duals of them (Proposition \ref{prp:gkz:as:polar}). Taking co-faces of the GKZ-conditions-cones corresponds exactly to coarsening the polytopal subdivision (Proposition \ref{prp:refined:coface}). Moreover, using GKZ-condition-cones, we define the notion of regular subdivisions (Definition \ref{def:regular:condition:cone}): it is shown that the definition coincides with the usual one when $N=1$. The higher rank GKZ-cones associated to regular subdivisions can be glued together, they yield a complete $\mu$-polyhedral fan in $M_{N,r}(\mathbb{R})$ (Theorem \ref{thm:gkz:complete} and \ref{thm:gkz:correspondence}). The questions (1) and (2) above have been fully answered so far.


\subsection{GKZ-cones, quasi-valuations and degenerations} 
In Section \ref{sec:quasi-valuations}, starting from a point $\Psi$ in the higher rank GKZ-cone $C(\mathcal{Q},N)$, we define a quasi-valuation $\mathcal{V}_\Psi$ with values in $\mathbb{Q}^N\cup\{\infty\}$ on the homogeneous coordinate ring $\mathbb{K}[X_A]$ of the toric variety $X_A$ associated to the point configuration $A$, embedded into the projective space $\mathbb{P}(\mathbb{K}^A)$ using points in $A$. The quasi-valuation $\mathcal{V}_{\Psi}$ comes from the construction of the regular polytopal subdivision associated to $\Psi$: their basic properties are studied in Section \ref{sec:quasi-valuations}.

The semi-toric degeneration obtained from the quasi-valuations $\mathcal{V}_\Psi$ is studied in Section \ref{two:graded:algebras}. Since $\mathcal{V}_\Psi$ is radical, the associated graded algebra $\mathrm{gr}_{\mathcal{V}_\Psi}\mathbb{K}[X_A]$ is reduced, moreover, it is finitely generated. Geometrically, $\mathrm{Proj}(\mathrm{gr}_{\mathcal{V}_\Psi}\mathbb{K}[X_A])$ is reduced, and it is an irredundant union of finitely many equidimensional toric varieties, one for each polytope in the polytopal subdivision $\mathcal{Q}$. Such a variety depends only on $\mathcal{Q}$ (Theorem \ref{coro:geometric1}). As a consequence, for any $\Psi\in M_{N,r}(\mathbb{R})$, the quasi-valuation
$\mathcal{V}_\Psi$ has a finite Khovanskii basis. This answers the question (3) above.

\subsection{Outline of the paper}

The discrete geometry of $\mu$-polyhedral cones is studied in Section \ref{sec:polyhedral:cones:and:polars}, where the order topology used there is recalled in Appendix \ref{order:topology}. The higher rank GKZ-cones and fans are defined and studied in Section \ref{sec:HiDiGKZ}. In Section \ref{sec:quasi-valuations} the quasi-valuation $\mathcal{V}_\Psi$ is defined from a matrix $\Psi$ in a GKZ-cone. The semi-toric degeneration arising from $\mathcal{V}_\Psi$ is studied in Section \ref{two:graded:algebras}.

\subsection{Acknowledgements}
The work of R.C. is partially supported by PRIN 2022 S8SSW2 ``Algebraic and geometric aspects of Lie theory" (CUP I53D23002410006). The work of M.C.C. is partially founded by PRIN 2022 A7L229 ``ALgebraic and TOPological combinatorics" (CUP J53D23003660006). The work of X.F. is funded by the Deutsche Forschungsgemeinschaft: “Symbolic Tools in Mathematics and their Application” (TRR 195, project-ID 286237555). The work of P.L. is partially supported by DFG SFB/Transregio 191 ``Symplektische Strukturen in Geometrie, Algebra und Dynamik''.

\subsection{Disclaimer}
The views expressed are purely those of the authors and may not in any circumstances be regarded as stating an official position of the European Research Council Executive Agency and the European Commission.

\subsection{\bf Tool and computational resource disclosure} During the preparation of this manuscript no AI tool has been used.

\section{Polyhedral cones and polars}\label{sec:polyhedral:cones:and:polars}

The goal of this section is to generalize the fundamental notion of the polar dual $P(K)\subseteq (\mathbb{R}^r)^*$ of a convex cone $K\subseteq\mathbb{R}^r$ to the framework which is appropriate for the study of higher rank GKZ-cones. Roughly speaking, the functions on the cone $K$ will be replaced by maps to $\mathbb R^N$ for some integer $N\geq 1$. For this we introduce $\mu$--polyhedral cones in the space of matrices, and establish a polar correspondence between polyhedral cones in $\R^r$ and $\mu$--polyhedral cones in $M_{N,r}(\mathbb R)$. In this context the co-faces appear naturally as the dual object of faces of $\mu$--polyhedral cones.

\subsection{Convex cones, co-faces and co-fans}

Let $r\geq 1$ be an integer and let $\mathbb R^r$ be a Euclidean vector space endowed with the standard scalar product $(\cdot,\cdot)$, i.e. for $\uv,\uw\in\mathbb R^r$ we set $(\uv,\uw):=\uv^T\cdot\uw$. For $\uu\in \mathbb{R}^r$ let $h_{ {\uu}}^\leq$ be the closed half-space $h_{ {\uu}}^\leq:=\{\uv \in \mathbb{R}^r\mid (\uu,\uv)\le 0\}$.

\begin{definition}\label{s2_polyhedral}\rm
A \emph{convex cone} in $\mathbb{R}^r$ is a subset $\mcC\subseteq \mathbb{R}^r$ such that for any $\uw_1,\uw_2\in \mcC$ we have $\uw_1+\uw_2\in \mcC$  and for any $\uw\in \mcC$, $\lambda\in \mathbb R_{\ge 0}$, $\lambda \uw\in \mcC$. It is called \emph{polyhedral} if there exist $\uu_1,\ldots,\uu_k\in\mathbb{R}^r$ such that $\mcC=h_{ {\uu_1}}^\leq\cap \cdots\cap h_{ {\uu_k}}^\leq$; a finite intersection of closed half-spaces. Or, equivalently, if there exist $\uv_1,\ldots,\uv_s$ in $\mcC$ such that
\[
\mcC=\langle \uv_1,\ldots, \uv_s\rangle_{\R_ {\geq 0}}:=\{\sum_{i=1}^s \lambda_i \uv_i \mid \lambda_1,\ldots,\lambda_s\ge 0\}.
\]
\end{definition}

On the family of convex cones we define two operations. The \emph{intersection} $\mcC_1\cap \mcC_2$ of two convex cones is a convex cone, and if both cones are polyhedral, so is the intersection. The \emph{sum} $\mcC_1 + \mcC_2 = \{\uv_1 + \uv_2 \mid  \uv_1\in \mcC_1, \uv_2 \in \mcC_2\}$ of convex cones is again a convex cone, and if both cones are polyhedral, so is the sum.

The tangent cone $T_\mcC(\uu)$ of a polyhedral cone $\mcC=h_{ {\uu_1}}^\leq\cap \cdots\cap h_{ {\uu_k}}^\leq$ at $\uu\in\mcC$ can be defined as
\[
T_{\mcC}(\uu) := \{ \uv \in \mathbb{R}^r\,|\, (\uv,\uu_j)\leq 0 \textrm{ for all } j \textrm{ such that } (\uu, \uu_j) = 0 \}.
\]

\smallskip

In this section, we introduce the notion of a \emph{co-face} of a convex cone in $\R^r$ and the notion of a \emph{co-fan}. As we will see below, these are convenient notions to establish a duality between the face structures of a cone and its polar dual in our context.

\begin{definition}\label{s2_def:coface}
A \emph{co-face} of a convex cone $\mcC\subseteq \mathbb{R}^r$ is a subset of $\R^r$ of the form $\mcC+\R\cdot\uu$ for some $\uu\in\mcC$.
\end{definition}

\begin{example}
Let $\mcC=\langle \uv_1, \uv_2\rangle_{\R_{\geq0}}$ be the cone in $\R^2$ generated by two linearly independent vectors $\uv_1$ and $\uv_2$. Then it is easy to check that the cofaces of $\mcC$ are: (1). $\mcC$ itself; (2). $\mcC+\R\cdot\uv_1 = \R\cdot\uv_1 + \R_{\geq0}\cdot\uv_2$; (3). $\mcC+\R\cdot\uv_2 = \R_{\geq0}\cdot\uv_1 + \R\cdot\uv_2$; (4). $\R^2$.
\end{example}

\begin{definition}
A finite collection $\mfS$ of polyhedral cones in $\R^r$ is called a \emph{co-fan} if the following two conditions hold:
\begin{itemize}
\item[(1)] if $\mcC\in\mfS$ and $\mcF$ is a co-face of $\mcC$ then $\mcF\in\mfS$;
\item[(2)] if $\mcC_1,\mcC_2\in\mfS$ then $\mcC_1 + \mcC_2$ is a co-face of both $\mcC_1$ and $\mcC_2$ (hence an element of $\mfS$). 
\end{itemize}
\end{definition}

\begin{proposition}\label{prp:coface-cofan}
Let $\mcC\subseteq\R^r$ be a polyhedral cone.
\begin{itemize}
\item[(i)] A subset of $\mathbb{R}^r$ is a co-face of $\mcC$ if and only if it is the tangent cone of $\mcC$ at some point of $\mcC$.
\item[(ii)] The collection of all co-faces of $\mcC$ is a co-fan. In particular, $\mcC$ has only finitely many co-faces.
\end{itemize}
\end{proposition}
\begin{proof}
Recall that, given $\uu\in\mcC$, the normal cone of $\mcC$ at $\uu$ is defined as
\[
N_\mcC(\uu) := \{ \uw \in \mathbb{R}^r\,|\, (\uw,\uv)\leq 0\textrm{ for all }\uv\in\mcC\textrm{ and }(\uv,\uu) = 0\}.
\]
For the purposes of this proof, we define the \emph{polar} of a convex cone $K\subseteq\R^r$ as $K^\circ := \{\ux \in \R^r\,|\, (\ux,\uy)\leq 0, \textrm{ for all }\uy\in K\}\subseteq\R^r$. It is well-known that (1) $T_\mcC(\uu) = \mcC^\circ\cap \uu^\perp$ and that (2) $N_\mcC(\uu) = T_\mcC(\uu)^\circ$.

It is then clear that $T_\mcC(\uu) = (\mcC^\circ\cap u^\perp)^\circ = \mcC + \R\cdot\uu$, which proves (i).

Moreover, the collection of all normal cones of $\mcC$ is a (finite) fan, called the normal fan of $\mcC$. Since polarity reverses inclusions and exchanges intersections and sums, we deduce that the collection of all co-faces of $\mcC$ is a co-fan, which proves (ii).
\end{proof}

\subsection{$\mu$-polyhedral cones}

We are going to generalize the notion of a closed half-space, which leads to a class of polyhedral-geometric objects, in the set of matrices $M_{N,r}(\mathbb{R})$, naturally showing up in the context of higher rank GKZ-cones.

To avoid confusion in the following, we reserve the notation ``$\,0\,$'' for the number $0$, and we use the notation ``$\,\mathbf 0\,$'' for a zero vector or a zero matrix.

Let $\mathbb R^N$ be endowed with the lexicographic order: that is, for $\bfa = (a_N,\ldots a_1)^T$ and $\bfb = (b_N,\ldots, b_1)^T$ in $\R^N$, we have $\bfa \leq \bfb$ if either $\bfa = \bfb$, or there exists $1\leq j\leq N$ such that $a_i = b_i$ for all $i > j$ and $a_j < b_j$. This order is compatible with addition and multiplication by non-negative scalars.

We have a natural  bilinear map for $M_{N,r}(\mathbb R)=\mathrm{Hom}_\mathbb R(\mathbb R^r,\mathbb R^N)$ and $\mathbb R^r$:
$$\mu: M_{N,r}(\mathbb R)  \times\mathbb R^r \rightarrow \mathbb R^N, \quad (\Psi,\uv)\mapsto \mu( \Psi,\uv):=\Psi\cdot\uv.$$
Here ``$\cdot$" denotes the matrix multiplication.

We consider the embedding of vector spaces $\mathbb{R}^r\to M_{N,r}(\mathbb{R})$ sending $\underline{v}\in\mathbb{R}^r$ to the matrix $\Phi_{\underline{v}}\in M_{N,r}(\mathbb{R})$ whose first row is $\underline{v}^T$ and all other entries are zero. For $\underline{w}\in\mathbb{R}^r$, 
\begin{equation}\label{Eq:Phi}
\Phi_{\underline{v}}\cdot \underline{w}=((\underline{v},\underline{w}),0,\ldots,0)^T\in\mathbb{R}^N.
\end{equation}

Guided by the usual definition of a closed half-space we define:

\begin{definition}\label{defn:half-space}\rm
The  \emph{generalized closed half-spaces} associated to $\Psi\in M_{N,r}(\mathbb R)$ and  $\uv\in \mathbb R^r$ are defined by:
$$
\mcH_{ {\Psi}}^\leq=\{\uv\in \mathbb R^r\mid \Psi\cdot\uv\le \mathbf 0\}\quad\textrm{and}\quad
H_{ {\uv}}^\leq =\{\Psi\in M_{N,r}(\mathbb R)\mid \Psi\cdot\uv\le \mathbf 0\}.
$$
In the same way we define the \emph{generalized open half-spaces} $\mcH_{ {\Psi}}^<=\{\uv\in \mathbb R^r\mid \Psi\cdot\uv < \mathbf 0\}$, $H_{{\uv}}^< =\{\Psi\in M_{N,r}(\mathbb R)\mid \Psi\cdot\uv < \mathbf 0\}$, and the subspace $H_{ {\uv}}^0 =\{\Psi\in M_{N,r}(\mathbb R)\mid \Psi\cdot\uv = \mathbf 0\}$.
\end{definition}

We note that all above subsets are convex cones in $M_{N,r}(\mathbb{R})$. Most of the time we \emph{omit} the term  \emph{generalized}, it will be clear from the context which kind of half-spaces are considered.

\begin{example}\label{classical}
For $N=1$ we get the usual closed half-spaces in $\mathbb R^r$. In particular, for $N=1$, the closed half-spaces are  polyhedral cones and they are closed in the Euclidean topology.
\end{example} 

\begin{example}\label{identity1}
Consider the case $r=N=2$. If $\Psi=\mathrm{id} \in M_{2,2}(\mathbb R)$ is the identity matrix, then $\mcH_{\mathrm{id}}^\leq\subseteq \mathbb R^2$ is the set of vectors $\uv=(v_2,v_1)^T$ such that either $v_2<0$, or $v_2=0$ and $v_1\le 0$. In particular, the cone  $\mcH_{\mathrm{id}}^\leq$ is neither a polyhedral cone, nor a closed subset in the Euclidean topology. The same arguments show that for $\underline{v}=(1,0)^T$,  $H_{\underline{v}}^\leq\subseteq  M_{2,2}(\mathbb R)$ is neither a polyhedral cone, nor a closed subset in the Euclidean topology.
\end{example} 

To justify the name \emph{closed} half-space for $H_{ {\uv}}^\leq$ not only by the analogy with the definition in the case $N=1$, we endow $\mathbb R^{N}$ with the order topology associated to the lexicographic order on $\mathbb R^N$ (see Appendix~\ref{order:topology} for details), and we endow $\mathrm{Hom}_\mathbb R(\mathbb R^r,\mathbb R^N)=M_{N,r}(\mathbb R)\simeq \R^N\times\cdots\times\R^N$ with the product of the order topologies. 
We call this topology the \emph{product order topology}. 

Adding vectors and multiplying a vector by a scalar are continuous operations  in $\R^N$ in the order topology. The following lemma in the matrix case can be found in Appendix~\ref{order:topology}, Lemma~\ref{lemma:cone:closed:open}. The statements on those in $\mathbb{R}^r$ can be proved similarly.


\begin{lemma}\label{Lem:Closed}
For $\underline v\in \mathbb R^r$ and $\Psi\in M_{N,r}(\mathbb{R})$, $H^0_{\uv}, H^\leq_{\uv}\subseteq M_{N,r}(\mathbb R)$ (resp. $\mathcal{H}^0_{\Psi}, \mathcal{H}^\leq_{\Psi}\subseteq \mathbb{R}^r$) are closed in the product order topology (resp. order topology);
and $H^<_{\uv}\subseteq M_{N,r}(\mathbb R)$ (resp. $\mathcal{H}^<_{\Psi}\subseteq \mathbb{R}^r$) is open in the product order topology (resp. order topology).
\end{lemma}

We generalize the notion of a polyhedral cone using the new class of closed half-spaces:

\begin{definition}\label{def:polyhedral:cone:vectors}\rm
A $\mu$-\emph{polyhedral cone} $C$ in $M_{N,r}(\mathbb R)$ is defined as a finite intersection of closed half-spaces ${C}=H_{ {\uv_1}}^\leq\cap\ldots\cap H_{ {\uv_s}}^\leq $ for some $\uv_1,\ldots, \uv_s\in \mathbb R^r$.
\end{definition}

Since closed half-spaces are cones, being their intersection, $\mu$-polyhedral cones in $M_{N,r}(\mathbb{R})$ are cones as well. 

It follows immediately from the definition and Lemma \ref{Lem:Closed}:

\begin{corollary}\label{cor:finite:intersection}
A finite intersection of $\mu$-polyhedral cones is a $\mu$-polyhedral cone. Such cones are closed in the product order topology of $M_{N,r}(\mathbb{R})$.
\end{corollary}


\subsection{Polar correspondence and $\mu$-polyhedral cones}

The polar dual of a polyhedral cone is a crucial notion in the study of its geometry and combinatorics. It plays a similar role in the study of $\mu$-polyhedral cones. In the $\mu$-polyhedral setup, the polar cone does not live in the same space as the cone itself, and the co-polar cone will be introduced to emphasize this difference.

\begin{definition}\label{def:polar}\rm
Let $\mcC\subseteq \mathbb R^r$ be a convex cone. We define its \emph{$N$-polar cone} as 
\[
P(\mcC)=\{\Psi\in M_{N,r}(\mathbb R)\mid \Psi\cdot \uv\le \mathbf 0 \text{ for all }\uv\in \mcC\}\,\subseteq M_{N,r}(\mathbb R).
\]
Similarly, the \emph{co-polar cone} of a convex cone $C\subseteq M_{N,r}(\mathbb R)$ is defined as
\[
\mathcal P(C)=\{\uv\in \mathbb{R}^r\mid \Psi\cdot \uv\le \mathbf 0\text{ for all } \Psi\in C\}\,\subseteq \mathbb{R}^r.
\]
\end{definition}

If $N$ is clear from the context, we will simply call $P(\mcC)$ the polar cone of $\mcC$.

\begin{lemma}\label{lem:polyhedral:cone2}
The $N$-polar cone $P(\mcC)$ of the polyhedral cone 
$ \mcC=\langle \uv_1,\ldots, \uv_k\rangle_{\R_ {\geq 0}}\subseteq \mathbb R^r$ is the $\mu$-polyhedral cone 
$\bigcap_{i=1}^k H_{\uv_i}^\leq\subseteq M_{N,r}(\mathbb R)$.
\end{lemma}
\begin{proof}
For a $\Psi\in M_{N,r}(\mathbb R)$ we have $\Psi\cdot\uv\leq \mathbf 0$ for all $\uv\in\mcC$ if and only if 
 $\Psi\cdot\uv_i\leq \mathbf 0$  for $i=1,\ldots,k$.  The  claim  follows immediately by the definitions.
\end{proof}

%

\begin{lemma}\label{lem:copolar:cone2}
The co-polar  cone $\mcP(C)$ of the $\mu$-polyhedral cone $C=\bigcap_{i=1}^k H_{\uv_i}^\leq$  is the polyhedral cone 
$\mcP(C)=\langle \uv_1,\ldots,\uv_k\rangle_{\R_{\geq 0}}\subseteq\mathbb R^r$.
\end{lemma}
\begin{proof}
For any $\Psi\in C$ and any $i=1,\ldots,k$ we have $\Psi\cdot\uv_i \leq\mathbf 0$;
hence $\uv_1,\ldots,\uv_k\in\mcP(C)$ and we find $\mcC:=\langle \uv_1,\ldots, \uv_k\rangle_{\R_ { \geq 0}}\subseteq\mcP(C)$.

For the other inclusion, suppose that there exists an element $\uw\in\mcP(C)$ such that $\uw\not\in\mcC$. By the hyperplane separation theorem for convex cones, there exists an element $\uu\in \mathbb R^r$ such that $(\uu,\uw)>0$ and $(\uu,\uv)\leq 0$ for all $\uv\in\mcC$.

Recall that for any vector $\uv\in \mathbb R^r$, we have $\Phi_{\uu}\cdot\uv = ((\uu,\uv),0,\ldots,0)^T$. 
In particular $\Phi_{\uu}\cdot \uv_i\leq\mathbf 0$ for all $i=1\,\ldots,k$, which implies $\Phi_{\uu}\in C$. But $(\uu,\uw)>0$ implies $\Phi_{\uu}\cdot \uw>\mathbf 0$ and hence $\uw\notin\mcP(C)$, which is a contradiction.
\end{proof}

The following is essential for the study of higher rank GKZ-cones later:

\begin{theorem}\label{thm:polar:copolar}
The map, which associates to a polyhedral cone  $\mcC\subseteq \mathbb R^r$ its $N$-polar cone $P(\mcC)$, and the map, which associates to a $\mu$-polyhedral cone $C\subseteq M_{N,r}(\mathbb R)$ 
its co-polar cone $\mcP(C)$, are inverse to each other. 
They induce an inclusion-reversing bijection between the set of polyhedral cones in $\mathbb R^r$ and
the set of $\mu$-polyhedral cones in $M_{N,r}(\mathbb R)$. 

With respect to this bijection, the sum of polyhedral cones in $\mathbb R^r$ corresponds to the intersection of $\mu$-polyhedral cones in $M_{N,r}(\mathbb R)$.
\end{theorem}

\begin{proof}
By Lemma~\ref{lem:polyhedral:cone2} and Lemma~\ref{lem:copolar:cone2}, the maps are bijections which are inverse to each other. They clearly reverse inclusions. Now let $C_1 = \bigcap_{i=1}^k H_{\uv_i}^\leq$ and  $C_2 = \bigcap_{i=1}^s H_{\uw_i,}^\leq$ be $\mu$-polyhedral cones in $M_{N,r}(\mathbb{R})$ where $\uv_1,\ldots,\uv_k,\uw_1,\ldots,\uw_s\in \mathbb R^r$. By Lemma \ref{lem:copolar:cone2} we get for the associated co-polar cones:
$$\mcC_1 := \mcP(C_1) = \langle \uv_1,\ldots,\uv_k\rangle_{\R_{\geq 0}},\quad\mcC_2 := \mcP(C_2) = \langle \uw_1,\ldots,\uw_s\rangle_{\R_{\geq 0}}.$$
For the intersection of the cones we get
$$\mcP(C_1\cap C_2) = \langle \uv_1,\ldots,\uv_k,\uw_1,\ldots,\uw_s\rangle_{\R_{\geq 0}} = \mcC_1 + \mcC_2$$
because $C_1\cap C_2 = \bigcap_{i=1}^k H_{\uv_i}^\leq \cap \bigcap_{i=1}^s H_{\uw_i}^\leq$. This shows that the sum of polyhedral cones in $\mathbb R^r$ corresponds to the intersection of $\mu$-polyhedral cones in $M_{N,r}(\mathbb R)$.
\end{proof}

We call the  bijection in the previous theorem the \emph{polar correspondence} between polyhedral
cones in $\mathbb R^r$ and  $\mu$-polyhedral cones in $M_{N,r}(\mathbb R)$.


\subsection{A dimension formula}

To a polyhedral cone $\mcC=h_{ {\uu_1}}^\leq\cap \cdots\cap h_{ {\uu_k}}^\leq \subseteq\R^r$ we associate  
three subspaces:
\begin{enumerate}
\item $L_\mcC := \mcC\cap(-\mcC)$ is the largest subspace of $\R^r$ contained in $\mcC$;
\item $U_\mcC\subseteq \R^r$ is the subspace spanned by $\uu_1,\ldots,\uu_k$;
\item $W_\mcC:=\{\Psi\in M_{N,r}(\R)\mid \Psi\cdot\uv = \mathbf 0 \text{ for all }\uv\in L_\mcC\}$.
\end{enumerate} 
It follows from definition that $P(\mcC)\subseteq  W_\mcC$. 

We write a row of $\Psi\in M_{N,r}(\R)$ as the transpose
$\uw^T$ of a vector $\uw\in \R^r$. The next lemma follows immediately from the definitions of the subspaces.

\begin{lemma}\label{lem:row:lemma}
\begin{itemize}
\item[{\rm (i)}] $L_\mcC$ is the orthogonal complement of $U_\mcC$ in  $\R^r$. 
\item[{\rm (ii)}]  $W_\mcC$ is the subspace of matrices such that the rows lie in $U_\mcC$,  i.e.
if $\uw^T_N,\ldots,\uw^T_1$ are the rows of $\Psi$, then $\Psi\in  W_\mcC$ if and only if $\uw_N,\ldots,\uw_1\in U_\mcC$.
\end{itemize}
\end{lemma}

Recall that the dimension of a cone is defined as the dimension of the real vector space spanned by the cone. Denote by $\overline{P(\mcC)}\subseteq M_{N,r}(\mathbb R)$ the closure (in the Euclidean topology) of the $\mu$-polyhedral cone 
$P(\mcC)$.

\begin{proposition}\label{prp:polar:dimension}
Let $\mcC=h_{ {\uu_1}}^\leq\cap \cdots\cap h_{ {\uu_k}}^\leq \subseteq\R^r$ be a polyhedral cone. Then:  
\begin{enumerate} 
\item[\rm (i)] $\dim P(\mcC) = N\cdot \codim L_\mcC=\dim \overline{P(\mcC)}$.
\item[\rm (ii)] $\overline{P(\mcC)}$ is a polyhedral cone in $M_{N,r}(\R)$. More precisely,
\begin{equation}\label{eq:closure:of:cone}
\overline{P(\mcC)}=\{\Psi\in W_\mcC\mid (\Psi_{N,1},\ldots,
\Psi_{N,r})\cdot\underline{u}_i\le 0 \text{ for all } i=1,\ldots,k\},
\end{equation} 
where $ (\Psi_{N,1},\ldots,\Psi_{N,r})$ is the top row of $\Psi$.
\end{enumerate}
\end{proposition}
\begin{proof}
It follows from the definition of the dimension of a cone that $\dim P(\mcC)=\dim \overline{P(\mcC)}$.
Denote the polyhedral cone on the right side of \eqref{eq:closure:of:cone} by $K_\mcC$,
we have clearly $P(\mcC)\subseteq K_\mcC$ and hence $ \overline{P(\mcC)}\subseteq K_\mcC$.

Lemma~\ref{lem:row:lemma} implies $\dim W_\mcC=N\cdot \codim L_\mcC$, and the image $\widetilde{\mcC}$ of $\mcC$ in $\mathbb R^r/L_\mcC\simeq U_\mcC$ is 
pointed. It follows: its polar cone
$$\mathfrak{P}(\widetilde{\mcC})=\{\uu\in U_\mcC\mid (\uu, \uv)\le 0 \text{ for all } \uv \in \widetilde{\mcC}\}$$
has full dimension and hence: 
\begin{itemize}
\item[\rm (I)] $\dim \mathfrak{P}(\widetilde{\mcC})=\dim U_\mcC$; 
\item[\rm (II)] there exists $\uu_0\in U_\mcC$ such that for any $\uv \in \widetilde{\mcC}\setminus\{\mathbf 0\}$, $(\uu_0, \uv) < 0$: in particular, for any $\uv \in \mcC\setminus L_\mcC$, $(\uu_0, \uv) < 0$.
\end{itemize}

The polyhedral cone $K_\mcC$ consists exactly of those matrices $\Psi$ in $W_\mcC$, such that the top row $\uw^T_N$ of $\Psi$ satisfies: $\uw_N\in \mathfrak{P}(\widetilde{\mcC})$ (note that by Lemma~\ref{lem:row:lemma}, $\underline{w}_N\in U_{\mathcal{C}}$). It follows from (I): $\dim K_\mcC=N\cdot \codim L_\mcC$. Further, consider the matrix $\Phi_{\uu_0}\in W_\mcC$ (see \eqref{Eq:Phi}  and (II)): if $\Psi\in K_\mcC$, then $\Psi+t\Phi_{\uu_0}\in P(\mcC)$ for all $t>0$, and hence $\Psi\in \overline{P(\mcC)}$. We have hence equality $K_\mcC=\overline{P(\mcC)}$, which finishes the proof.
\end{proof}


\subsection{Faces and fans for $\mu$-polyhedral cones}

We start with the definition of a face of a $\mu$-polyhedral cone in $\MNaR$ in our context.

\begin{definition}\label{def:face}
A \emph{face} of a $\mu$-polyhedral cone $C$ in $\MNaR$ is a subset of $C$ obtained as an intersection $C\cap H_{\uu}^0$ for some $\uu\in \mathcal{P}(C)$.
\end{definition}

In the next proposition, we see that the polar correspondence transforms faces of $\mu$-polyhedral cones in $\MNaR$ into co-faces of polyhedral cones in $\R^r$. This correspondence is the key to understand the geometry of $\mu$-polyhedral cones in $\MNaR$.
\begin{proposition}\label{prp:face:coface}
\begin{itemize}
\item[(i)] Let $C$ be a $\mu$-polyhedral cone in $\MNaR$ and $\uu\in\mathcal{P}(C)$. Then the face 
$F = C\cap H_{\uu}^0$ is a $\mu$-polyhedral cone and $\mcP(F) = \mcP(C) + \R\cdot\uu$ is a co-face of $\mcP(C)$.
\item[(ii)] Let $\mcC$ be a polyhedral cone in $\R^r$ and $\uu\in\mcC$. Then $P(\mathcal{C})\subseteq H_{\underline{u}}^{\leq}$, the co-face $\mcF=\mcC + \R\cdot\uu$ is a polyhedral cone in $\R^r$ and $P(\mcF) = P(\mcC)\cap H_{\uu}^0$ is a face of $P(\mcC)$. 
\item[(iii)] Faces correspond to co-faces under the polar correspondence (Theorem \ref{thm:polar:copolar}).
\item[(iv)] A face of a face is a face.
\item[(v)] A $\mu$-polyhedral cone has finitely many faces.
\end{itemize}
\end{proposition}

\begin{proof}
For the part \emph{(i)}: Since $F = C\cap H_{\uu}^0 = C\cap H_{\uu}^\leq\cap H_{-\uu}^\leq$, $F$ is a $\mu$-polyhedral cone. It follows from Theorem \ref{thm:polar:copolar} that 
$$\mcF := \mcP(F) = \mcP(C) + \R_{\geq 0}\cdot\uu + \R_{\geq 0}(-\uu) = \mcP(C) + \R\cdot\uu.$$
Hence $\mcF$ is a co-face of $\mcP(C)$. 

The proof of \emph{(ii)} is similar. Furthermore \emph{(iii)} follows by combining \emph{(i)} and \emph{(ii)}. Claims \textrm{(iv)} and \textrm{(v)} follow by \textrm{(iii)} since the set of co-faces is a co-fan (Proposition \ref{prp:coface-cofan}) and by the polar correspondence (Theorem \ref{thm:polar:copolar}).
\end{proof}

It is clear that a polyhedral cone is not the union of its proper faces. 
In order to prove the same result for $\mu$-polyhedral cones in $\MNaR$ we need some preliminary results.

\begin{proposition}\label{prp:dimension:proper}
    \begin{itemize}
        \item[(i)] If $\mcF$ is a proper co-face of a polyhedral cone $\mcC$ in $\R^r$ then $\dim L_\mcF > \dim L_\mcC$.
        \item[(ii)] If $F$ is a proper face of a $\mu$-polyhedral cone $C$ in $\MNaR$ then $\dim F < \dim C$.
    \end{itemize}
\end{proposition}
\begin{proof}
    The statement \emph{(ii)} follows  by \emph{(i)} by the polar correspondence and Proposition~\ref{prp:polar:dimension}. For \emph{(i)} let $\mcF = \mcC + \R\cdot\uu$, with $\uu\in\mcC$. Then $L_\mcC\subseteq L_\mcF$ since $\mcC\subseteq\mcF$. Hence $\dim L_\mcF\geq\dim L_\mcC$.
    
    If we had $\dim L_\mcF = \dim L_\mcC$, then $L_\mcF = L_\mcC$. But since $\R\cdot\uu\subseteq\mcF$, we have $\uu\in L_{\mcF}$. So we got also $\uu\in L_\mcC$ and $\R\cdot\uu\subseteq L_\mcC\subseteq\mcC$. So we conclude $\mcF = \mcC + \R\cdot\uu\subseteq\mcC$, but this is impossible since $\mcF$ is a proper co-face.
\end{proof}

\begin{proposition}\label{prp:cone:not:union}
A $\mu$-polyhedral cone in $\MNaR$ is not the union of its proper faces.
\end{proposition}
\begin{proof} Suppose that the $\mu$-polyhedral cone $C$ is the union of its proper faces which, by Proposition \ref{prp:face:coface}, are $\mu$--polyhedral cones and finite in number: $C = \bigcup_{i = 1 }^k C_i$ with $C_1,\ldots, C_k$ the proper faces of $C$. It follows for the closure (with respect to the Euclidean topology) that $\overline{C}\subseteq \bigcup_{i = 1 }^k \overline{C_i}$.
    
By Theorem~\ref{thm:polar:copolar}, $C$, $C_1$, \ldots, $C_k$ are the polar cones of polyhedral cones in $\mathbb R^r$. So, by Proposition~\ref{prp:polar:dimension} \textrm{(ii)}, $\overline{C}$ and the closures $\overline{C_i}$, $i = 1,\ldots,k$, are polyhedral cones. But, by Proposition~\ref{prp:dimension:proper} and Proposition~\ref{prp:polar:dimension} \textrm{(i)} we have $\dim \overline{C_i} = \dim C_i < \dim C = \dim \overline{C}$ for all $i=1,\ldots,k$. It follows that $\overline{C}\subseteq \bigcup_{i = 1 }^k \overline{C_i}$ is impossible.
\end{proof}

We are interested in fans of $\mu$-polyhedral cones in $\MNaR$.

\begin{definition}
A finite collection $\bfSigma$ of $\mu$-polyhedral cones in $\MNaR$ is called a \emph{fan} if the following two conditions hold:
\begin{itemize}
\item[(1)] if $C\in\bfSigma$ and $F$ is a face of $C$ then $F\in\bfSigma$;
\item[(2)] if $C_1,C_2\in\bfSigma$ then $C_1\cap C_2$ is a face of both $C_1$ and $C_2$ (hence an element of $\bfSigma$). 
\end{itemize}
Moreover, a fan is \emph{complete} if $\bigcup_{C\in\bfSigma}C = \MNaR$.
\end{definition}

The polar correspondence induces a duality between fans and co-fans.

\begin{theorem}\label{thm:fan:cofan}
If $\bfSigma$ is a fan of $\mu$-polyhedral cones in $\MNaR$ then $\mfS=\{\mcP(C) \mid C\in\bfSigma\}$ is a co-fan of polyhedral cones in $\R^r$. On the other hand, if $\mfS$ is a co-fan of polyhedral cones in $\R^r$ then $\bfSigma = \{P(\mcC) \mid \mcC\in\mfS\}$ is a fan of $\mu$-polyhedral cones in $\MNaR$.
\end{theorem}

\begin{proof} 
Both statements follow from Theorem \ref{thm:polar:copolar} and Proposition \ref{prp:face:coface}.
\end{proof}


\section{Higher dimensional GKZ-cones}\label{sec:HiDiGKZ}

Throughout this section fix integers $n,N\geq 1$.

\subsection{Polytopes and functions}\label{subsec:polytopes:functions}
 
We start with a finite set $A=\{v_1,\ldots,v_r\}\subseteq \mathbb R^{n-1}$. Without loss of generality we assume that the affine span of $A$ is $\mathbb R^{n-1}$. Denote by $P := \mathrm{conv}(A)$ the polytope obtained as the convex hull of $A$. The pair $(P,A)$ is called a \emph{marked polytope}.

Since we often switch between the vector spaces $\mathbb R^A\cong\mathbb{R}^r$, $\mathbb R^n$ and $\mathbb R^N$, in this section, we keep the notation $\uu, \uv, \ldots$ for column vectors in $\mathbb R^A$ and  $\mathbf a, \mathbf b, \ldots$ for column vectors in $\mathbb R^N$, and $v,w,\ldots\in \R^n$  respectively $\mathbb R^{n-1}$.


\begin{definition}\label{def:polytopal:subdivision}
By a \emph{polytopal subdivision $\mathcal Q=(Q_i,A_i)_{i\in I}$ of $(P, A)$ by marked polytopes
$(Q_i,A_i)$, $i\in I$}, we mean a decomposition of $P$ into a finite number of $(n-1)$-dimensional 
polytopes $Q_i$, $i\in I$, such that $P=\bigcup_{i\in I} Q_i$, $A_i\subseteq A\cap Q_i$ for all $i\in I$, 
and for $i,j\in I$, $i\not=j$: 
the intersection $Q_i\cap Q_j$ is either empty or a face of both,  and $A_i\cap Q_i\cap Q_j=A_j\cap Q_i\cap Q_j$. 
\end{definition}

Note that the $Q_i$, $i\in I$, form a polytopal complex. We do not require every element of $A$ to appear as a marking of one of the polytopes in $\mathcal Q$.

\begin{definition}\label{def:triangulation}
A \emph{triangulation $\mathcal T$ of $(P, A)$} is a polytopal subdivision of $(P, A)$ by marked polytopes $\mathcal Q=(Q_i,A_i)_{i\in I}$ such that all $i\in I$: $Q_i$ is a simplex and $A_i$ is the set of vertices of $Q_i$.
\end{definition}

\begin{example}\label{example:simplex:r2}
Set $A=\{(0,0), (3,0), (0,3), (1,1)\}\subseteq \mathbb R^2$ and let  $P$ be the simplex obtained as convex hull of $A$. 
The marked polytope $(P,A)$ has two triangulations by marked polytopes: $\mathcal Q_1=(P,A')$, where $A'\subseteq A$
are the vertices of $P$, and $\mathcal Q_2=(Q_i,A_i)_{i=1,2,3}$, where the $Q_i$ are the triangles obtained
by joining the point $(1,1)$ with two of the vertices of $P$, and $A_i$ is the set of vertices of $Q_i$. There is one more
polytopal subdivision: $\mathcal Q_0=(P,A)$. Though $P$ is a simplex, this is not a triangulation
because $(1,1)\in A$ is not a vertex of $P$.
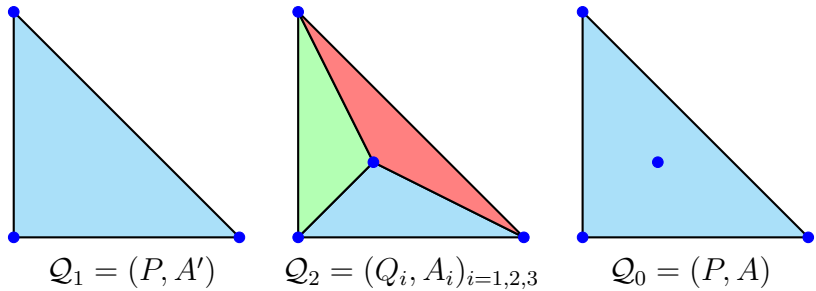
\[
\begin{array}{ccc}
\begin{tikzpicture}
    \draw[thick] [fill=cyan!30] (0,0) coordinate(A) -- (3,0) coordinate(B) -- (0,3)coordinate(C) -- cycle;
    \filldraw[blue] (3,0) circle (2pt);
    \filldraw[blue] (0,0) circle (2pt);
    \filldraw[blue] (0,3) circle (2pt);
\end{tikzpicture}\ 
&\ 
\begin{tikzpicture}
    \draw[thick] [fill=cyan!30] (0,0) coordinate(A) -- (3,0) coordinate(B) -- (1,1)coordinate(C) -- cycle;
    \draw[thick]  [fill=green!30] (0,0) coordinate(A) -- (0,3) coordinate(B) -- (1,1)coordinate(C) -- cycle; 
    \draw[thick]  [fill=red!50](3,0) coordinate(A) -- (0,3) coordinate(B) -- (1,1)coordinate(C) -- cycle; 
    \filldraw[blue] (3,0) circle (2pt);
    \filldraw[blue] (0,0) circle (2pt);
    \filldraw[blue] (1,1) circle (2pt);
    \filldraw[blue] (0,3) circle (2pt);
\end{tikzpicture}\ 
&\ 
\begin{tikzpicture}
    \draw[thick] [fill=cyan!30] (0,0) coordinate(A) -- (3,0) coordinate(B) -- (0,3)coordinate(C) -- cycle;
    \filldraw[blue] (3,0) circle (2pt);
    \filldraw[blue] (0,0) circle (2pt);
    \filldraw[blue] (0,3) circle (2pt);
    \filldraw[blue] (1,1) circle (2pt);
\end{tikzpicture}\\
\mathcal Q_1=(P,A')& \mathcal Q_2=(Q_i,A_i)_{i=1,2,3} & \mathcal Q_0=(P,A)\\
\end{array}
\]
\end{example}

We define a relation between the polytopal subdivisions of $(P,A)$.

\begin{definition}\label{def:refine:subdivision}
Let  $\mathcal Q=(Q_i,A_i)_{i\in I}$,  $\mathcal Q'=(Q'_j,A'_j)_{j\in J}$ be polytopal subdivisons of $(P,A)$ by marked polytopes. We say that \emph{$\mathcal Q$ refines $\mathcal Q'$\,} if, for each $j\in J$, the collection of $(Q_i, A_i)$ such that $Q_i\subseteq Q'_j$ forms a polytopal subdivision of $(Q_j', A_j')$ by marked polytopes. 
\end{definition}
This makes the set of all polytopal subdivisons of $\mathcal Q$ into a poset such that
triangulations are minimal elements of this poset and $(P,A)$ is the unique maximal element.

\begin{example}
In Example~\ref{example:simplex:r2} both $\mathcal Q_1$ and $\mathcal Q_2$ refine $\mathcal Q_0$.
\end{example}

Since we only discuss polytopal subdivisons  by marked polytopes, 
we often omit in the following the addition \emph{by marked polytopes}.
Similarly, if $(Q,A_Q)\in \mathcal Q$ is a pair occurring in a polytopal subdivison, 
we often forget about the accompanying subset $A_Q\subseteq A$ and just write: 
let $Q\in\mathcal Q$ be a polytope.

\smallskip

Via the embedding $\mathbb R^{n-1}\hookrightarrow \mathbb R^n$, $u\mapsto (1,u)$, we identify $\mathbb R^{n-1}$ with an affine hyperplane in $\mathbb R^n$, and we identify $P$ with its image in $\mathbb R^n$. To avoid an unnecessary accumulation of notation, we write
$v$ for elements in $\R^{n-1}$ and $\R^n$. In particular, the element of $A$ are also denoted by $u, w$  both as elements of $\R^{n-1}$ 
and as elements of  $\R^n$, it will be clear from the context where they are.
The \textit{cone $K(P)$ associated to $P$} is the cone in $\mathbb R^n$ generated by the polytope $P$, it is a polyhedral cone.
 
\begin{definition}\label{definition:piecewise:linear}
A \textit{$\mathcal Q$--piecewise-linear map} $g: K(P) \rightarrow  \mathbb R^N$ is a map which is linear on the cones $K(Q)$, $Q\in\mcQ$.
\end{definition}

Notice that a map $g:K(P)\longrightarrow\R^N$ is linear on $K(Q)$ if and only if
\begin{itemize}
\item[(i)] $g$ is affine on $Q$, i.e. $g(t u + (1-t) v) = tg(u) + (1-t)g(v)$ for each $u, v\in Q$ and $0\leq t\leq 1$; and
\item[(ii)] $g$ is homogeneous, i.e. $g(\lambda  {u}) = \lambda g( {u})$ for each $\lambda\in\R_{\geq 0}$ and $ {u}\in Q$.
\end{itemize}

The set of $\mathcal Q$-piecewise-linear maps is naturally an $\R$--vector space.
Recall that on $\mathbb{R}^N$ we consider the lexicographic ordering $\geq$. 

\begin{definition}\label{def:convex:upward}
Let $K\subseteq \mathbb R^n$ be a convex subset.
A continuous map $g:K\rightarrow \mathbb R^N$ is called \emph{convex upward} if for all $u, v\in K$ and $t\in[0,1]$ holds:  
\begin{equation}\label{definition:convex}
g(tu+(1-t)v) \geq tg(u)+(1-t)g(v).
\end{equation} 
\end{definition}

 \begin{remark}\label{convex:equivalent}\rm
For a $\mathcal Q$-piecewise-linear map $g:K(P)\rightarrow \mathbb R^N$, the homogeneity and the convex upward property imply: for all 
$u, v\in K(P)$ we have

\begin{equation}\label{remark:equivalence:convex}
g(u+v) \ge  \frac{1}{2}g(2u)+\frac{1}{2}g(2v)=g(u)+g(v).
\end{equation} 
One can show: the conditions \eqref{definition:convex} and \eqref{remark:equivalence:convex} are equivalent for $\mathcal Q$-piecewise-linear maps.
\end{remark}

The set of convex upward maps $g:K(P)\to\mathbb{R}^N$ forms a cone in the vector space of $\mathcal{Q}$-piecewise-linear maps. Moreover, a homogeneous map $g:K(P)\longrightarrow\R^N$ is convex upward if and only if its restriction to $P$ is convex upward.

Let $\mathcal Q=(Q_i,A_i)_{i\in I}$ be a polytopal subdivision of  $(P, A)$ and let $g:K(P)\rightarrow \mathbb R^N$ be a $\mathcal Q$-piecewise-linear map. If $Q\in \mathcal Q$ is a polytope, then $\{v \mid v\in A_Q\}$ is a spanning set for $\mathbb R^n$. So by definition, the restriction $g\vert_{K(Q)}$ of a $\mathcal Q$-piecewise-linear map to the cone $K(Q)$ coincides with the restriction $\psi_Q\vert_{K(Q)}$ of a uniquely determined linear map $\psi_Q:\mathbb R^n\rightarrow \mathbb R^N$. This leads to the following characterization of the convex upward-property:

\begin{proposition}\label{prop:convex:up:=:minimum}
Let $\mathcal Q=(Q_i,A_i)_{i\in I}$ be a polytopal subdivision of  $(P, A)$ and let $g:K(P)\rightarrow \mathbb R^N$ be a $\mathcal Q$-piecewise-linear map. The following are equivalent:
\begin{enumerate}
\item[\rm (i)] $g$ is convex upward;
\item[\rm (ii)] for all $u\in K(P)$, $g(u)=\min\{\psi_{Q}( u)\mid Q\in \mathcal Q\}$, where $\min$ is taken with respect to the lexicographic order on $\mathbb R^N$.
\end{enumerate}
\end{proposition}

\begin{proof} Suppose that $g$ is convex upward, let $Q\in\mathcal{Q}$ be a polytope and let $u\in K(P)$.

Since $A_Q$ is a spanning set for $\mathbb{R}^n$, we may write $u = u_+- u_-$ with both $u_+$ and $u_-$ in $K(Q)$. Hence by Remark \ref{convex:equivalent}, $\psi_Q(u_+)=g(u+u_-)\geq g(u) + g(u_-) = g(u) + \psi_Q(u_-)$ and we conclude $g(u)\leq \psi_Q(u_+) - \psi_Q(u_-) = \psi_Q(u)$.

But there must exist a polytope $Q'\in\mathcal Q$ such that $u\in K(Q')$, so $g(u) = \psi_{Q'}( u)$ and we have proved that $g$ is the minimum as stated in \emph{(ii)}.

On the other hand, suppose that $g$ is the minimum as in \emph{(ii)} and let $  u,  u'\in K(P)$. Then we have
$g(u + u')  = \min\{\psi_Q(  u) + \psi_Q(  u')\,|\,Q\in\mathcal Q\} $, and hence:
$$g(u + u')   \geq    \min\{\psi_Q(  u)\,|\,Q\in\mathcal Q\} + \min\{\psi_Q(  u')\,|\,Q\in\mathcal Q\}=g(u) + g(u'),$$
so the map $g$ is convex upward.
\end{proof}

Next we give a description to the linearity regions of a $\mcQ$--piecewise-linear and convex upward map. Set 
$A_{\mathcal Q}:=\bigcup_{i\in I} A_i$: it is the subset of $A$ containing all the vertices of the subdivision. 
By a \emph{domain $D$} contained in $P$ we mean a closed domain, i.e. $D$ is the closure 
(in the Euclidean topology) of a non-empty connected open subset in $P$.


\begin{definition}\label{def:domain:linearity}
A \emph{domain of linearity} for a $\mathcal{Q}$-piecewise-linear and convex upward map $g:K(P)\to\R^N$ is a domain $D\subseteq P$ 
such that $g_{|K(D)}$ is linear, and $D$ is maximal with this property.
\end{definition}

\begin{lemma}\label{lem:domain:linearity}
If $g:K(P)\to\R^N$ is a $\mcQ$--piecewise-linear and convex upward map,
then the domains of linearity of $g$ are polytopes with vertices in $A_\mcQ$. More precisely,
each domain of linearity is a convex set which is a union of polytopes in $\mcQ$.
\end{lemma}

\begin{proof} Let $D$ be a domain of linearity of $g$ and let $\varphi:\R^n\longrightarrow \R^N$ be a linear map such that $g_{|K(D)} = \varphi_{|K(D)}$. Since $D$ contains a non-empty open subset and the polytopes in $\mathcal Q$ are full dimensional, there exists a polytope $Q\in\mathcal Q$ whose open part intersects non-trivially the open part in $D$. It follows $\varphi = \psi_Q$.

Let $D'= \{ x\in P\,|\, g(x) = \psi_Q(x)\}$, clearly $D\subseteq D'$. We show: $D'$ is a convex set. 
If $x,y\in D'$ and $0\leq t\leq 1$, then, $g$ being convex upward, we have
\begin{equation}\label{equation:domain:linearity}
g(t {x}+(1-t) {y})\geq t g( {x})+(1-t)g( {y})=t\psi_Q( {x})+(1-t)\psi_Q( {y})=\psi_Q(t {x}+(1-t) {y}).
\end{equation}
Proposition \ref{prop:convex:up:=:minimum} \textit{(ii)} implies that we have equality everywhere in \eqref{equation:domain:linearity}, 
thus $tx + (1-t)y\in D'$ for all $t\in[0,1]$, and the convexity of $D'$ follows.

The subset $D'$ is closed and convex, and its interior is non-empty since it contains $Q$.
So $D'$ is the closure of an open convex (in particular connected) set, hence $D'$ is a domain. 
But $g$ is linear on $K(D')$ so $D= D'$ by the maximality of $D$.

It remains to show that $D$ is a union of polytopes in $\mathcal Q$.  Take $ {y}\in D\setminus Q$.
The convexity of $D$ implies that the convex hull $C:=\mathrm{conv}(Q\cup\{y\})$ is a convex set contained in $D$.
Since $Q$ is a full dimensional polytope, $D$ as well as $C$ are full dimensional convex sets. In particular,
given an open neighborhood $U$ of $ {y}$ (viewed as element in $\mathbb R^{n-1}$), then $U$ meets the relative interior of $C$.
Let $Q_1,\ldots,Q_t$ be the polytopes in $\mathcal Q$ containing $ {y}$. By choosing $U$ small enough, we can assume that $U\cap P$
is the same as the union of the $Q_i\cap U$, $i=1,\ldots,t$. Since $U$ meets the relative interior of $C$, there must exist $1\le j\le t$ such that 
$Q_j$ meets the relative interior of $C$, and hence $Q_j\cap C$ is full dimensional. It follows that the equality
$\psi_{Q_j}\vert_{Q_j\cap C}=\psi_{Q}\vert_{Q_j\cap C}$ implies $\psi_{Q_j}\vert_{Q_j}=\psi_{Q}\vert_{Q_j}$. But this implies $Q_j\subseteq D$.

Thus we have shown: given $ {y}\in D\setminus Q$, there exists a polytope $Q'\in\mathcal Q$ such that $ {y}\in Q'$ and 
$Q'\subseteq D$. But this implies that $D$ is a union of polytopes in $\mcQ$.
Now $D$ is both a convex set and a union of polytopes, hence $D$ itself is a polytope.
\end{proof}

\begin{corollary}\label{cor:linearity:only:in:domain}
    Let $v_1,\ldots,v_m \in K(P)$. Then $g(v_1 + \cdots + v_m) = g(v_1) + \cdots + g(v_m)$ if and only if $v_1,\ldots,v_m$ are contained in the cone over a common domain of linearity of $g$.
\end{corollary}
\begin{proof}
The ``if'' part is clear. To prove the ``only if'' part, let $D$ be a domain of linearity such that $v_1 + \cdots + v_m\in K(D)$. We claim that $v_1,\ldots,v_m\in K(D)$.

By Lemma \ref{lem:domain:linearity}, there exists $Q\in \mathcal Q$ such that $D$ is the set of all $v\in P$ such that $g(v) = \psi_Q(v)$. Hence $g(v_1 + \cdots + v_m) = \psi_Q(v_1 + \cdots + v_m) = \psi_Q(v_1) + \cdots + \psi_Q(v_m)$. For each $i=1,\ldots,m$, we have $g(v_i)\leq \psi_Q(v_i)$ by Proposition \ref{prop:convex:up:=:minimum}, so the equality $g(v_1 + \cdots + v_m) = g(v_1) + \dots + g(v_m)$ implies $g(v_i) = \psi_Q(v_i)$ for each $i=1,\ldots,m$. Hence $v_i\in K(D)$ for each $i=1,\ldots,m$.
\end{proof}


\subsection{The GKZ-cones}

Let $\mathcal Q=(Q_i,A_i)_{i\in I}$ be a polytopal subdivision of the marked polytope $(P, A)$ with $A=\{v_1,\ldots,v_r\}$ and $\Phi$ be a matrix in $\Psi\in\MNaR$. For $\ell=1,\ldots,r$ we denote by $\Psi(v_\ell)\in\mathbb{R}^N$ the $\ell$--th column of $\Psi$. 

\begin{definition}\label{def:gkz:close:cone}
    The \emph{(closed) GKZ-cone} $\overline{C(\mcQ,N)}\subseteq \MNaR$ is defined as the set of matrices $\Psi$ for which there exists a map $g:K(P)\to\R^N$ such that
    \begin{itemize}
        \item[(1)] $g$ is a $\mcQ$--piecewise-linear and convex upward;
        \item[(2)] for all $v\in A_\mcQ$ we have $g(v) = \Psi(v)$;
        \item[(3)] for all $v\in A\setminus A_\mcQ$ we have $g(v) \geq \Psi(v)$. 
    \end{itemize}
\end{definition}

The map $g$ in the definition is in fact unique.

\begin{proposition}\label{prp:unique:Q-linear}
Let $\Psi\in\overline{C(\mcQ,N)}$. The associated map $g$ in the previous definition is unique. 
We denote it by $g_{\Psi,\mcQ}$.
\end{proposition}

\begin{proof} 
Let $g:K(P)\to\R^N$ be as in Definition \ref{def:gkz:close:cone}. Each $Q\in\mcQ$ is a marked polytope, in particular $Q$ is the convex hull of $A_Q\subseteq A_\mcQ$. Since $g$ is linear on each $Q\in\mcQ$, its value at each point of $Q$ is uniquely determined by the values $g(v) = \Psi(v)$, 
$v\in A_Q$. But the polytopes in $\mcQ$ cover $P$, hence $g$ is unique.
\end{proof}
\begin{remark}\rm
It is easy to check that for $\lambda\in\mathbb R_{\ge 0}$
and $\Psi,\Phi\in \overline{C(\mcQ,N)}$ one has: $\lambda g_{\Psi,\mcQ}=g_{\lambda\Psi,\mcQ}$
and $g_{\Psi,\mcQ}+g_{\Phi,\mcQ}=g_{\Psi+\Phi,\mcQ}$, so $ \overline{C(\mcQ,N)}$ is indeed a cone.
\end{remark}
The following definition is justified by Lemma \ref{lem:domain:linearity}.

\begin{definition}\label{def:subdivision:matrix}
For $\Psi\in\overline{C(\mcQ,N)}$, we denote by $\mcD_{\Psi,\mcQ}=(D_{j},A_{j})_{j\in J}$ the 
    polytopal subdivision of $(P,A)$ where
    \begin{itemize}
    \item[-] the collection $(D_j)_{j\in J}$ of polytopes  is given by the domains of linearity of $g_{\Psi,\mcQ}$; 
    \item[-] the marking $A_j$ of $D_j$, $j\in J$, is defined by: $A_j=\{ v\in A\cap D_j\mid g_{\Psi,\mcQ}(v) = \Psi(v)\}$.
    \end{itemize}
\end{definition}

The following lemma follows immediately from the definitions and Proposition \ref{prp:unique:Q-linear}.

\begin{lemma}\label{lem:Q:refine}
    \begin{itemize}
        \item[\rm (i)] If $\mcQ$ refines $\mcQ'$, then $\overline{C(\mcQ',N)}\subseteq\overline{C(\mcQ,N)}$.
        \item[\rm (ii)] If $\Psi\in\overline{C(\mcQ,N)}$, then $\mcQ$ refines $\mcD_{\Psi,\mcQ}$.
        \item[\rm (iii)] If $\Psi\in\overline{C(\mcQ',N)}$ and $\mcQ$ refines $\mcQ'$, then $g_{\Psi,\mcQ'} = g_{\Psi,\mcQ}$.
    \end{itemize}
\end{lemma}

Next we study a decomposition of a GKZ-cone.

\begin{definition}\label{def:gkz:open:cone}
    The \emph{(open) GKZ-cone} $C(\mcQ,N)\subseteq \overline{C(\mcQ,N)}$ consists of $\Psi\in\overline{C(\mcQ,N)}$ such that $\mcD_{\Psi,\mcQ} = \mcQ$. 
\end{definition}

In other words, $\Psi\in \overline{C(\mcQ,N)}$ is contained in $C(\mcQ,N)$ if and only if 
    \begin{itemize}
        \item[(1)] the domains of linearity of $g_{\Psi,\mcQ}$ are exactly the polytopes in $\mcQ$;
        \item[(2)] for all $v\in A\setminus A_\mcQ$ we have $g_{\Psi,\mcQ}(v) > \Psi(v)$.
    \end{itemize}

In the following we write $\mcQ'>\mcQ$ if $\mcQ$ refines $\mcQ'$ but $\mcQ\not =\mcQ'$.

\begin{proposition}\label{prp:gkz:open:closed}
We have:
    \[
    C(\mcQ,N) \,\, = \,\, \overline{C(\mcQ,N)}\,\,\setminus\,\,\bigcup_{\mcQ'>\mcQ}\overline{C(\mcQ',N)}.
    \]
\end{proposition}

\begin{proof}
If $\Psi$ is a matrix in the open GKZ-cone $C(\mcQ,N)$, then $\Psi\in\overline{C(\mcQ,N)}$ by definition. Suppose $\mcQ$ refines $\mcQ'$. If $\Psi\in\overline{C(\mcQ',N)}$ then $g_{\Psi,\mcQ'} = g_{\Psi,\mcQ}$ by Lemma \ref{lem:Q:refine} \emph{(iii)}. But $\mcD_{\Psi,\mcQ} = \mcQ$, by the definition of open GKZ-cone. Hence, by Lemma \ref{lem:Q:refine} \emph{(ii)}, $\mcQ'$ refines $\mcD_{\Psi,\mcQ'} = \mcD_{\Psi,\mcQ} = \mcQ$. So $\mcQ' = \mcQ$ and we have proved that $\Psi\not\in\overline{C(\mcQ',N)}$ for any $\mcQ'>\mcQ$.

On the other hand, let $\Psi\in\overline{C(\mcQ,N)}$ be such that $\Psi\notin\overline{C(\mcQ',N)}$ for any $\mcQ'>\mcQ$. 
By Lemma~\ref{lem:Q:refine} \emph{(ii)},  $\Psi\in\overline{C(\mcQ,N)}$ implies $\mcQ$ refines $\mcQ':=\mcD_{\Psi,\mcQ}$. 
Since $\Psi\notin\overline{C(\mcQ',N)}$ if $\mcQ'>\mcQ$, we must have 
$\mcQ' = \mcQ$. So $\Psi\in C(\mcQ,N)$ since $\mcD_{\Psi,\mcQ} = \mcQ$.
\end{proof}

We will need the following observation in the proof of the next Theorem:

\begin{lemma}\label{lemma:max:in:polytope}
Let  $R\subseteq \mathbb R^{N}$ be a polytope. There exists a unique point 
$\mathbf p_R\in R$ such that $\mathbf p_R>\mathbf p'$ (with respect to the lexicographic order) 
for all $\mathbf p'\in R\setminus\{\mathbf p_R\}$. Moreover, 
$\mathbf p_R$ is a vertex of the polytope.
\end{lemma}

We call $\mathbf p_R$ the maximal element of $R$ and write $\mathbf p_R=\max\{\mathbf p'\mid \mathbf p'\in R\}$.

\begin{proof}
The claim holds if $R$ is a point or $R$ is an interval in $\mathbb R$. Suppose now
$\dim R> 1$. Let $x_N:\mathbb R^N\rightarrow \mathbb R$ be the projection onto the first coordinate.
Since $R$ is compact and $x_N$ is continuous, the restriction $x_N\vert_R$ attains its maximum $\mathfrak m_N$ 
in $R$, and $R$ is contained in the half space $\{\mathbf u\in \mathbb R^N\mid x_N(\mathbf u)\le \mathfrak m_N\}$.
It follows that the set of points in $R$ where $x_N\vert_R$ attains its maximum is a face $F$ of $R$.

The  polytope $F$ lives in $\{\mathbf u\in \mathbb R^N\mid x_N(\mathbf u)=\mathfrak m_N\}\cong\mathbb R^{N-1}$ as affine spaces, with $\dim F\le \dim R$ so proceeding by induction over $N$, $F$ has a unique maximal point $\mathbf  p:=\max\{\mathbf  p'\mid \mathbf p'\in F\}$, which is a vertex of $F$. By construction, $\mathbf  p$ is also maximal for $R$, and a vertex of the face $F$ is also a vertex of the polytope $R$.
\end{proof}

As the main result of this section, we show that the (open) GKZ-cones cover $M_{N,r}(\mathbb{R})$.

\begin{theorem}\label{thm:gkz:complete}
For each $\Psi\in \MNaR$ there exists a polytopal subdivision $\mcQ$ of $(P,A)$ such that $\Psi\in C(\mcQ,N)$.
\end{theorem}

\begin{proof}
In the following we construct a piecewise linear map $g_\Psi:P\rightarrow \mathbb R^N$ with the desired properties. The corresponding piecewise linear map $g_\Psi:K(P)\rightarrow \mathbb R^N$ is obtained via linear extension.

For $\Psi\in M_{N,r}(\mathbb R)$ let $\Omega_\Psi\subseteq \mathbb R^{n-1}\oplus \mathbb R^N$ be the polytope obtained as the convex hull:
\begin{equation}\label{equation:Omega:Psi}
\Omega_\Psi:=\textrm{convex hull of\ }\{(  v,\Psi(  v))\mid   v\in A\}\subseteq \mathbb R^{n-1}\oplus \mathbb R^N.
\end{equation}
Let $p_1: \mathbb R^{n-1}\oplus \mathbb R^N\rightarrow \mathbb R^{n-1}$ be the canonical projection onto the first component. It is an affine map such that $p_1(\Omega_\Psi)=P$. For $w\in P$, the fibre of the restricted projection 
\begin{equation}\label{equation:fibre:Omega:w}
\Omega_{ {w}}:=(p_1\vert_{\Omega_\Psi})^{-1}(w) ,
\end{equation}
can be identified with a polytope in $\mathbb R^N$. By Lemma~\ref{lemma:max:in:polytope},  the map
\begin{equation}\label{definition:gPsi}
g_{\Psi}: P\rightarrow \mathbb R^N,\quad   w\mapsto \max \Omega_{ w}
\end{equation}
is well-defined.

Let $F$ be a face of the polytope $\Omega_\Psi$ and denote by $F^o$ its relative interior. We call $F$ an \emph{upper face} if for all $(w,\mathbf u)\in F$ holds: $\mathbf u=g_\Psi(  w)$. 

\begin{claim}
A face $F$ of $\Omega_\Psi$ is an upper face if and only if there exists $w\in P$ such that $(w,g_{\Psi}(w))\in F^o$.
\end{claim}

\begin{proof}
If $F$ is a point, then the conditions are clearly equivalent. Suppose $\dim F\ge 1$. Note that $\dim p_1(F)\ge 1$ because the vertices of $F$ are of the form $(w_1,\Psi(w_1))$, $\ldots$, $(w_s,\Psi(w_s))$ for some $w_1,\ldots, w_s\in A$, where $s\ge 2$.

If $F$ is an upper face, then by definition, $\mathbf u=g_{\Psi}(w)$ for all $( w,\mathbf u)\in F$. 

On the other hand, let $(  w,g_{\Psi}(  w))$ be a point in the relative interior $F^o\subset F$. We argue by contradiction. Suppose there exists a point $( y, \mathbf u)\in F$, $ y\not= w$, and an element 
$\mathbf u'\in \mathbb R^N$ such that $\mathbf u'>\mathbf u$ and $(  y, \mathbf u')\in \Omega_\Psi$. 
The ray starting in $(  y,\mathbf u)$ and passing through $(  w,g_{\Psi}( w))$ 
meets a proper face $F'\subset F$ of $F$. 
Let $( u_1,\mathbf u_1),\ldots,(  u_k,\mathbf u_k)$ be the vertices of $F'$. It follows that 
$(  w,g_{\Psi}(  w))$ is a convex linear combination of these vertices and $(  y,\mathbf u)$:
\[
(  w,g_{\Psi}(  w))= \lambda (  y,\mathbf u)+\lambda_1(  u_1,\mathbf u_1)+\ldots +\lambda_k(  u_k,\mathbf u_k),
\]
where $\lambda\not=0$. But note that $\lambda (  y,\mathbf u')+\lambda_1( u_1,\mathbf u_1)+\ldots +\lambda_k(  u_k,\mathbf u_k)=:(  w,\mathbf x)$ 
is an element in $\Omega_\Psi$ with $\mathbf x>g_{\Psi}(  w)$, which is not possible. It follows: $(  y,\mathbf u)
=(  y,g_{\Psi}(  y))$ for all $(  y,\mathbf u)\in F$, $  y\not=  w$. But now one can replace $(  w,g_{\Psi}(  w))$ 
by any other element $(  y,g_{\Psi}(  y))\in F^o$, $  y\not=  w$, and the same arguments show: for any
$(  w,\mathbf u)\in F$, $\mathbf u=g_{\Psi}(  w)$, \emph{i.e.} $F$ is an upper face.
\end{proof}

\par
\noindent
\textit{Continuation of the proof of Theorem~\ref{thm:gkz:complete}.}
Let $\mathcal F$ be the union of the upper faces. The restriction $\pi\vert_{\mathcal F}:\mathcal F\rightarrow P$ of the projection is injective by construction. For the surjectivity: take $w\in P$, the point $(w,g_{\Psi}( w))$ is contained in the boundary of the polytope $\Omega_\Psi$. Hence it is a point in the relative interior $F^o$ of a unique face $F$ of $\Omega_\Psi$. By the above claim, the face is hence an upper face, and 
$  w\in \mathrm{im\,}p_1$.

By definition, the face of an upper face is an upper face, so the images of the maximal dimensional 
upper faces $\{F_j\}_{j\in J}$ define a subdivision $\{Q_j\}_{j\in J}$ of $P$.
The restriction $g_\Psi\vert_{Q_j}$ is affine linear, it is the inverse map to the (bijective) projection 
$p_1\vert_{F_j}: F_j\rightarrow Q_j$. So $g_{\Psi}$ is piecewise linear, and, by construction, its domains of linearity are 
exactly the polytopes $Q_j$, $j\in J$. It remains to add the marking of the polytope: we set 
$A_j=\{  v\in A\cap Q_j\mid  g_\Psi( v)=\Psi({v})\}$. The set $A_j$ contains by construction all the vertices of $Q_j$. 
It follows that $\mathcal Q=(Q_j,A_j)_{j\in J}$ is a polytopal subdivision of $(P, A)$ such that $g_{\Psi}=g_{\Psi,\mathcal Q}$, 
the domains of linearity are exactly the polytopes in $\mathcal Q$, and it follows from the definition of each $A_j$ that $g_{\Psi,\mathcal Q}(v)>\Psi(v)$ 
for all $v\in A\setminus A_{\mathcal Q}$, so $\Psi\in C(\mathcal Q, N)$.
\end{proof}

\subsection{The GKZ-cones and the polar correspondence}

In this section, using polar correspondence, we prove that the GKZ-cones are $\mu$-polyhedral cones. To simplify notations, instead of $\mathbb{R}^r$, we will sometimes use $\mathbb{R}^A$, forgetting the index of elements in $A$.

Let $\{\ue_v\mid v\in A\}$ denote the canonical basis of $\R^A$. For a matrix $\Psi\in\MNaR$ and a vector 
$\uu = \sum_{v\in A}a_v\ue_v\in\R^A$, the multiplication of the matrix $\Psi$ and the vector $\uu$ can be written as $\Psi\cdot\uu = \sum_{v\in A}a_v\Psi(v)$.

Let $\bfDelta = \{w_1,\ldots,w_n\}\subseteq A$ be a subset such that $\bfDelta$ is a basis for $\mathbb R^n$. It is at the same time an affine basis for the affine hyperplane $\R^{n-1}\subseteq \R^n$ spanned by $P$, where the latter is embedded via $w\mapsto (1,w)$. We call $\bfDelta$ for short an \emph{affine basis}. 

Given a vector $v\in A$ we define the vector
$$\uu_{v,\bfDelta} := \ue_v - (a_1\ue_{w_1} + \cdots + a_n\ue_{w_n}) \in \R^A$$
where $v = a_1w_1 + \cdots + a_nw_n$ is the (unique) expression of $v$ as a linear combination of the affine basis $\bfDelta$.

\begin{definition}\label{def:condition:cone}
Let $\mcQ$ be a polytopal subdivision of $(P,A)$. The \emph{GKZ-condition-cone} $\mcC_\mcQ$ is the polyhedral cone in $\R^A$ generated by the vectors
$$\uu_{v,\bfDelta},\,-\uu_{v,\bfDelta}\quad\textrm{with }\bfDelta\subseteq A_Q\textrm{ an affine basis, }Q\in \mcQ,\,\, v\in A_Q$$
and by the vectors
$$\uu_{v,\bfDelta}\quad\textrm{with }\bfDelta\subseteq A_Q\text{ an affine basis, }Q\in \mcQ,\,\,  v\in A\setminus A_Q.$$
\end{definition}

\begin{example} 
We keep the notations in Example \ref{example:simplex:r2}:  set $A = \{\uv_1, \uv_2, \uv_3, \uv_4\}$ with $\uv_1 = (0,0)$, $\uv_2 = (3,0)$, $\uv_3 = (0,3)$ and $\uv_4 = (1,1)$. We abbreviate $\ue_i := \ue_{\uv_i}$ for $i=1,2,3,4$. Since $\uv_4 = \frac{1}{3}\uv_1 + \frac{1}{3}\uv_2 + \frac{1}{3}\uv_3$, we let $\uu := \ue_4 - \frac{1}{3}\ue_1 - \frac{1}{3}\ue_2 - \frac{1}{3}\ue_3\in\R^A$.

Consider the subdivsion $\mathcal Q_1 = (P,A')$ with $A' = \{\uv_1, \uv_2, \uv_3\}$. Since $A'$ is an affine basis, $\uu = \uu_{v_4, \bfDelta}$ is a generator of the GKZ-condition-cone $\mcC_{\mcQ_1}$, \emph{i.e.} $\mcC_{\mcQ_1} = \R_{\ge 0}\cdot \uu$.

For the polytopal subdivision $\mcQ_2$ the vertices of each simplex $Q_1,\,Q_2,\,Q_3$ form an affine basis. Assuming that $Q_i$ is the simplex not containing $v_i$, we have $\uu_{v_i,Q_i} = -3\uu$ for $i=1,2,3$. Hence $\mcC_{\mcQ_2} = -\R_{\geq 0}\cdot \uu = -\mcC_{\mcQ_1}$.

The GKZ-condition-cone $\mcC_{\mcQ_0}$ is generated by $\uu$ and $-\uu$, hence $\mcC_{\mcQ_0} = \R\cdot \uu$.

Note that $\mcQ_0$ is a common coface of $\mcQ_1$ and $\mcQ_2$ (see Proposition \ref{prp:refined:coface} below).
\end{example}

\noindent{\bf Key fact.} The GKZ-condition-cone $\mcC_\mcQ$ is \emph{independent} of $N$, it depends \emph{only} on the affine geometry of the polytopes in the subdivision $\mcQ$.
\smallskip

The set of generators for $\mcC_\mcQ$ given in the definition is far away from being minimal. A simple calculation shows:

\begin{lemma}\label{Lem:BaseChange}
Let $\bfDelta:=\{ w_1,\ldots,w_n\}\subseteq A_Q$ 
and $\bfDelta':=\{ w_1',\ldots,  w_n'\}\subseteq A_Q$ be two affine bases.
For $v\in A$ with $ \uu_{v,\bfDelta'}:=\ue_v -\sum_{i=1}^n b_i\ue_{w_i'}$, we have:
$$ \uu_{v,\bfDelta'}=\uu_{v,\bfDelta}-\sum_{j=1}^n b_j\uu_{w_j',\bfDelta}.$$
\end{lemma}

The sign rule in the definition above implies that the vectors $\pm \uu_{  w_j,\bfDelta}$, $1\le j\le n$,
are among the generators of the GKZ-condition-cone $\mcC_\mcQ$.
So $ \uu_{v,\bfDelta'}$ is in the subcone of $\mcC_\mcQ$ spanned by $\uu_{v,\bfDelta}$ and  $\pm\uu_{w_j',\bfDelta}$,
$1\le j\le n$.

To get a smaller generating set for $\mcC_\mcQ$, it suffices to fix for every 
polytope $Q\in\mathcal Q$ an affine basis $\bfDelta$ and take the generators
$\uu_{v,\bfDelta}$ (and $-\uu_{v,\bfDelta}$ if appropriate) with respect to this basis.

The poset of polytopal subdivisions gets the following interpretation into that of co-faces.

\begin{proposition}\label{prp:refined:coface}
If $\mcQ$ refines $\mcQ'$ then $\mcC_{\mcQ'}$ is a co-face of the cone $\mcC_\mcQ$.
\end{proposition}


\begin{proof}

For a polytope $Q\in\mathcal{Q}$ (resp. $Q'\in\mathcal{Q'}$), set $\mathcal{B}_Q$ (resp. $\mathcal{B}_{Q'}$) to be the set of all affine bases in $A_Q$ (resp. $A_{Q'}$). Fix polytopes $Q\in\mathcal{Q}$ and $Q'\in\mathcal{Q}'$ with $Q\subseteq Q'$, we denote
$$\mathcal{E}_{Q,Q'}:=\{\underline{u}_{w_j',\bfDelta}\mid \bfDelta':=\{w_1',\ldots,w_n'\}\in\mathcal{B}_{Q'},\ \bfDelta\in\mathcal{B}_{Q}\},$$
and let $\mathcal{E}$ be the union of $\mathcal{E}_{Q,Q'}$ where $Q\in\mathcal{Q}$ and $Q'\in\mathcal{Q}'$ with $Q\subseteq Q'$.

We will show that $\mathcal{C}_{\mathcal{Q}}+\langle\mathcal{E}\rangle_{\mathbb{R}}=\mathcal{C}_{\mathcal{Q}'}$. The proposition then follows by Proposition \ref{prp:face:coface} \emph{(v)}.

For the inclusion $\subseteq$, take $Q\in\mathcal{Q}$ and $Q'\in\mathcal{Q}'$ with $Q\subseteq Q'$, then $\mathcal{B}_Q\subseteq\mathcal{B}_{Q'}$. It follows immediately that the vectors in the GKZ-condition-cone arising from $Q\in\mathcal{Q}$ form a subset of those arising from $Q'\in\mathcal{Q}'$. This shows $\mathcal{C}_{\mathcal{Q}}\subseteq\mathcal{C}_{\mathcal{Q}'}$. Let $\bfDelta\in\mathcal{B}_Q$ and $\bfDelta'=\{w_1',\ldots,w_n'\}\in\mathcal{B}_{Q'}\setminus\mathcal{B}_Q$. Since $w_j'\in A_{Q'}$, $\pm \underline{u}_{w_j',\bfDelta}$ are vectors in GKZ-condition cones arising from $Q'$, and hence $\langle\mathcal{E}\rangle_{\mathbb{R}}\subseteq\mathcal{C}_{Q'}$.

For the other inclusion, according to the discussion after Lemma \ref{Lem:BaseChange}, up to the linear space $\langle\mathcal{E}\rangle_{\mathbb{R}}$, for each polytope $Q'\in\mathcal{Q}'$ we can fix for every $Q'\in\mathcal{Q'}$ an affine basis $\bfDelta'$ which is also an affine basis for some $Q\in\mathcal{Q}$. Moreover, up to this linear space, for the vectors $\pm\underline{u}_{v,\bfDelta}$, it suffices to consider this $\bf\Delta'\subseteq A_Q$, and $v\in A_Q$. We have, up to this linear space, reduced all vectors for $\mathcal{C}_{\mathcal{Q}'}$ to those for $\mathcal{C}_{\mathcal{Q}}$.
\end{proof}

As the main result of this section, we show that the GKZ-cones are $\mu$-polyhedral.

\begin{proposition}\label{prp:gkz:as:polar}
Let $\mcQ$ be a polytopal subdivision of $(P,A)$. Then $\overline{C(\mcQ,N)} = P(\mcC_\mcQ)$. In particular, the closed GKZ-cones are $\mu$-polyhedral cones in $\MNaR$, and hence closed in the product order topology.
\end{proposition}

\begin{proof}
Let $\bfDelta=\{w_1,\ldots,w_n\}\subseteq A_Q$, $Q\in\mcQ$, be an affine basis, $v\in A_Q$ and $v=\sum_{i=1}^n a_iw_i$ be the (unique) expression of $v$ in $\bfDelta$. We have for $\Psi\in\MNaR$:
$$\Psi\cdot\uu_{v,\bfDelta} = \Psi(v) - \sum_{i=1}^n a_i\Psi(w_i).$$
Hence $\Psi$, looked as a function from $A$ to $\mathbb{R}^N$, can be extended to a linear map on $K(Q)$ if and only if $\Psi\cdot\uu_{v,\bfDelta} = \mathbf 0$ for all $ v\in A_Q$. Therefore we have a well-defined $\mathcal Q$-piecewise linear map $g_{\Psi,\mathcal Q}:K(P)\rightarrow \mathbb R^N$ if and only if the two conditions  $\Psi\cdot\uu_{v,\bfDelta} \leq \mathbf 0$ and $\Psi\cdot(-\uu_{v,\bfDelta}) \leq \mathbf 0$ are fulfilled for any $Q\in\mathcal{Q}$, any affine basis $\bfDelta\subseteq A_Q$ and $v\in A_Q$.
    
The next point is to view the convex upwards property in terms of the cone of GKZ-conditions $\mcC_\mcQ$. We use the characterization of the convex upwards property given in Proposition~\ref{prop:convex:up:=:minimum}. Since the linear maps $\psi_Q$, $Q\in\mcQ$, are completely
determined by the values $\Psi(  v)$, $  v\in A_Q$, we see that $g_{\Psi,\mathcal Q}$ is convex upwards if and only if for all $Q\in \mathcal Q$ and all $v\in A_\mcQ\setminus A_Q$, say $v\in Q'$, we have $\psi_{Q'}(v)=\Psi(v)\le \psi_{Q}(v)$. Let $\bfDelta=\{w_1,\ldots,w_n\} \subseteq A_Q$ be an affine basis and write $v=\sum_{i=1}^n a_iw_i$, then 
$$\Psi({v})=\psi_{Q'}(v)\leq\psi_{Q}(v)=\sum_{i=1}^n a_i\Psi(w_i) \text{ if and only if }\Psi\cdot\uu_{v,\bfDelta} \le \mathbf 0.$$
Therefore the conditions $\Psi\cdot \uu_{v,\bfDelta}\le 0$ with $\bfDelta\subseteq A_Q$ an affine basis, $Q\in \mcQ$,  $ v\in A_\mcQ\setminus A_Q$ are equivalent to the fact that $g_{\Psi,\mathcal Q}$ is convex upwards.
    
Suppose $v\in A\setminus A_\mcQ$, recall that $g_{\Psi,\mathcal{Q}}(v)=\min\{\psi_{Q}(  v)\mid Q\in \mathcal Q\}$. Fix $Q\in \mcQ$ and an affine basis $\mathbf{\Delta}\subseteq A_Q$ as above, it follows for $v=\sum_{i=1}^n a_iw_i$ that $\psi_{Q}(v)= \sum_{i=1}^n a_i\Psi(w_i)$. With similar calculation, the condition $g_{\Psi,Q}(v)\ge \Psi(v)$ for all  $v\in A\setminus A_\mcQ$ is equivalent to  $\Psi\cdot \uu_{v,\bfDelta}\le 0$ for all $Q\in \mcQ$, $\bfDelta\subseteq A_Q$ an affine basis and $ v\in A \setminus A_\mcQ$.

Summarizing we have: $\Psi\in\overline{C(\mcQ,N)}$ if and only if $\Psi\in P(\mcC_\mcQ)$.
\end{proof}

Since the closed GKZ-cones are $\mu$-polyhedral cones, it follows immediately:

\begin{corollary}\label{coro:elementary:moves}
Let $\Psi\in {C(\mcQ,N)}$. If $\Phi\in M_{N,r}(\mathbb R)$ is the matrix obtained from 
$\Psi$ by one of the following elementary moves, then $\Phi\in {C(\mcQ,N)}$ too:
\begin{enumerate}
\item a row of $\Psi$ is multiplied by a positive real number;
\item a multiple of a row of $\Psi$  is added to another row below;
\item a row of $\Psi$ is changed by adding the same real number to all entries in that row.
\end{enumerate}
\end{corollary}

\begin{proof}
The first two claims follow immediately from Proposition~\ref{prp:gkz:as:polar} together with 
Proposition~\ref{prp:gkz:open:closed} because these operations preserve the equalities and the strict inequalities. To prove {\it 3)}, recall that  $\Psi\cdot\uu_{v,\bfDelta} = \Psi(v) - \sum_{i=1}^n a_i\Psi(w_i)$, where $\bfDelta=\{w_1,\ldots,w_n\}$ is an affine basis and hence $\sum_{i=1}^n a_i=1$. It follows $\Psi\cdot\uu_{v,\bfDelta} =\Phi\cdot\uu_{v,\bfDelta} $, which proves the corollary.
\end{proof}

\subsection{Regular subdivisions}

In this section we introduce and study regular polyhedral subdivisions, which are higher rank generalizations of the coherent subdivisions in \cite{GKZ}.

\begin{definition}\label{def:regular:condition:cone}
A polytopal subdivision $\mcQ$ of $(P,A)$ is called \emph{regular} if the cone $\mcC_\mcQ$ is properly contained in each $\mcC_{\mcQ'}$ with $\mcQ'>\mcQ$, i.e. $\mcQ$ refines $\mcQ'$ and $\mcQ'\not=\mcQ$.
\end{definition}

The notion of regularity is clearly independent of $N$. Proposition~\ref{prp:gkz:open:closed} together with Proposition~\ref{psp:regular} imply that for $N=1$, the definition above coincides with the classical definition in \cite[Chapter 7, Definition 2.3]{GKZ}, where the term \emph{coherent} is used instead of \emph{regular}.

\begin{proposition}\label{psp:regular}
The polytopal subdivision $\mcQ$ is regular if and only if $C(\mcQ,N)$ is not the empty set.
\end{proposition}

\begin{proof}
If $\mcQ$ is not regular, then there exists a polytopal subdivision $\mcQ'$ such that $\mcQ'>\mcQ$ and $\mcC_\mcQ = \mcC_{\mcQ'}$. Hence $\overline{C(\mcQ,N)} = P(\mcC_\mcQ) = P(\mcC_{\mcQ'}) = \overline{C(\mcQ',N)}$; so $C(\mcQ,N) = \varnothing$ by Proposition \ref{prp:gkz:open:closed}.

Suppose that $\mcQ$ is regular, \emph{i.e.}, for every polytopal subdivision $\mcQ'>\mcQ$ we have $\mcC_\mcQ \subsetneq \mcC_{\mcQ'}$. So, by Proposition \ref{prp:refined:coface}, each $\mcC_{\mcQ'}$ is a \emph{proper} co-face of $\mcC_\mcQ$. Then, by the polar correspondence (Theorem \ref{thm:polar:copolar}), $P(\mcC_{\mcQ'}) = \overline{C(\mcQ',N)}$ is a proper face of $P(\mcC_\mcQ) = \overline{C(\mcQ,N)}$. Thus, the open cone $C(\mcQ,N)$ is non-empty, since by Proposition \ref{prp:cone:not:union}, the polyhedral cone $\overline{C(\mcQ,N)}$ is not the union of  its proper faces.
\end{proof}

Let $\mfS$ be the collection of the GKZ-condition-cones $\mcC_\mcQ$, where $\mcQ$ runs over all regular subdivisions of $(P,A)$.

\begin{proposition}\label{prp:condition:cofan}
The collection $\mfS$ is a co-fan of polyhedral cones in $\R^A$. The cones $\mcC_\mcQ$, with $\mcQ$ a regular polytopal subdivision of $(P,A)$, are pairwise different. The co-faces of $\mcC_\mcQ$ are the cones $\mcC_{\mcQ'}$ with $\mcQ'$ a (regular) polytopal subdivision such that $\mcQ$ refines $\mcQ'$.
\end{proposition}

\begin{proof}
Applying Proposition \ref{prp:gkz:as:polar} to $N=1$ shows that the collection $\mfS$ is the co-polar of the collection of closed GKZ-cones $\overline{C(\mcQ,1)}$, where $\mcQ$ runs over regular polytopal subdivisions. In \cite[Chapter 7, Proposition 1.5]{GKZ} it is proved that this collection of GKZ-cones is a fan of polyhedral cones in $M_{1,r}(\R)\simeq\R^A$. They also proved that the faces of the cone $\overline{C(\mcQ,1)}$ are the GKZ-cones $\overline{C(\mcQ',1)}$ such that $\mcQ$ refines $\mcQ'$ and $\mcQ'$ is regular. It follows from this description of the faces that these GKZ-cones are all different.

Now it remains to apply Theorem \ref{thm:polar:copolar} to get all the stated properties of the co-fan $\mfS$ in $\R^A$.
\end{proof}

Note that, by the definition of a regular subdivision, the collection $\mfS$ is also the collection of all the GKZ-condition-cones. Indeed the condition-cone $\mcC_\mcQ$ for a non-regular $\mcQ$ is equal to a condition-cone $\mcC_{\mcQ'}$ with $\mcQ'$ a regular polytopal subdivision (which is refined by $\mcQ$).

Let $\bfSigma_N$ be the collection of the GKZ-cones $\overline{C(\mcQ,N)}$ where $\mcQ$ runs over all regular polytopal subdivision of $(P,A)$.

\begin{theorem}\label{thm:gkz:correspondence}
    The fan $\bfSigma_N$ and the co-fan $\mfS$ are in polar correspondence. In particular the collection $\bfSigma_N$ of the GKZ-cones $\overline{C(\mcQ,N)}$, where $\mcQ$ runs over all regular polytopal subdivision of $(P,A)$, is a fan of $\mu$-polyhedral cones in $\MNaR$. These cones are pairwise different. The faces of the cone $\overline{C(\mcQ,N)}$ are the cones $\overline{C(\mcQ',N)}$ with $\mcQ'$ a regular polytopal subdivision such that $\mcQ$ refines $\mcQ'$. Moreover the fan $\bfSigma_N$ is complete.
\end{theorem}

\begin{proof} 
All the statements but the last one follow immediately by Proposition \ref{prp:gkz:as:polar} and Proposition \ref{prp:condition:cofan}. The last statement follows by Theorem \ref{thm:gkz:complete}.
\end{proof}

We discuss some consequences of the theorem.

\begin{corollary}\label{cor:gkz:dimension}
The dimension of the GKZ-cone $\overline{C(\mcQ,N)}$ is 
$$\dim\overline{C(\mcQ,N)} = N\codim L_{\mcC_\mcQ}=N\cdot \dim \overline{C(\mcQ,1)}.$$
\end{corollary}

\begin{proof} 
The dimension formula follows from Proposition \ref{prp:polar:dimension} and Theorem \ref{thm:gkz:correspondence}.
\end{proof}

\begin{corollary}\label{cor:open:partition}
\begin{itemize}
\item[\rm (i)] If $\mcQ$ is regular, then the open GKZ-cone $C(\mcQ,N)$ is open and dense in $\overline{C(\mcQ,N)}$ with respect to the product order topology.
\item[\rm (ii)] If $\mcQ$ is regular, then the dimension of $C(\mcQ,N)$ is equal to that of $\overline{C(\mcQ,N)}$.
\item[\rm (iii)] The collection of open GKZ-cones $C(\mcQ,N)$, where $\mcQ$ is a regular polytopal subdivision of $(P,A)$, forms a partition of $M_{N,r}(\mathbb{R})$.
\end{itemize}
\end{corollary}

\begin{proof}
The open cone $C(\mcQ,N)$ is obtained by removing all proper faces from $\overline{C(\mcQ,N)}$. By Proposition~\ref{prp:gkz:open:closed}  and Theorem~\ref{thm:gkz:correspondence}, there are only  finitely many of them. Moreover these faces are closed subsets in the product order topology and their union, being finite in number, is closed. Hence $C(\mcQ,N)$ is open in $\overline{C(\mcQ,N)}$. 

For the density, we start from the following claim.

\begin{claim}
For any neighborhood $U_{\mathbf{0}}$ of $\mathbf{0}$ in the product order topology, $U_{\mathbf{0}}\cap C(\mcQ,N)\neq\emptyset$.
\end{claim}

\begin{proof}
Consider the inclusion of vector spaces $\iota:M_{1,r}(\mathbb R)\hookrightarrow M_{N,r}(\mathbb R)$, $\Psi'\mapsto\Psi$, where $\Psi$ is obtained from $\Psi'$ by adding $(N-1)$-zero rows on the top of $\Psi'$. The notion of regularity is independent of $N$, so we know $C(\mcQ,1)\neq\emptyset$, and $\iota$ induces an inclusion $\iota: C(\mcQ,1)\hookrightarrow C(\mcQ,N)$. Further, for $1\leq \ell\leq r$, let $a'_\ell$, $b'_\ell\in\mathbb{R}$ be such that $a'_\ell<0<b'_\ell$. The product $(a_1',b_1')\times\ldots\times(a'_\ell,  b'_\ell)\subseteq M_{1,r}(\mathbb R)$ is a neighborhood of the zero matrix, and its image under $\iota$ is $(\mathbf a_1,\mathbf b_1)\times \ldots\times (\mathbf a_\ell,\mathbf b_\ell)$ where $\mathbf{a}_\ell=(0,\ldots,0,a_\ell)^t$, $\mathbf b_\ell=(0,\ldots,0,b_\ell)^t$ and $(\mathbf{a}_\ell,\mathbf{b}_\ell):=\{\mathbf{c}\in\mathbb{R}^N\mid \mathbf{a}_\ell<\mathbf{c}<\mathbf{b}_\ell\}$. This latter product is a neighborhood of $\mathbf{0}\in M_{N,r}(\mathbb R)$ in the product order topology. 

For a neighborhood $U_{\mathbf 0}$ of $\mathbf{0}\in M_{N, r}(\mathbb R)$ one can always find a product of intervals $(a_1,b_1)\times\ldots\times (a_\ell,b_\ell)\in M_{1, r}(\mathbb R)$ whose image under $\iota$ is contained in $U_{\mathbf 0}$. The map $\iota$ reduces hence the proof of the claim to the case $N=1$. But in this case the claim holds because the Euclidean and the product order topology coincide for $N=1$.
\end{proof}

Once this has been established, take any $\Psi\in \overline{C(\mcQ,N)}$ and any of its neighborhood $U$, we need to show that $U\cap C(\mcQ,N)\neq\emptyset$. Since the addition by a fixed matrix is a homeomorphism in the product order topology (Lemma \ref{Lem:TopGroup}), $U_{\mathbf{0}}:=U-\Psi$ is a neighborhood of $\mathbf{0}$. It follows from the above claim that $U_{\mathbf{0}}\cap C(\mcQ,N)\neq\emptyset$. As $\overline{C(\mcQ,N)}+C(\mcQ,N)\subseteq C(\mcQ,N)$, $U\cap C(\mcQ,N)\neq\emptyset$. This completes the proof of \emph{(i)}.

Each of the finite proper faces removed from $\overline{C(\mcQ,N)}$ has dimension strictly less than those of $\overline{C(\mcQ,N)}$ by Proposition \ref{prp:dimension:proper}. So it is clear that $C(\mcQ,N)$ has the same dimension as $\overline{C(\mcQ,N)}$ (compare with the proof of Proposition \ref{prp:cone:not:union}), which finishes the proof of \emph{(ii)}.

The last point \emph{(iii)} follows by Proposition \ref{prp:gkz:open:closed} and the fact that $\bfSigma_N$ is complete.
\end{proof}

\begin{remark}\rm
The first part of the corollary justifies the notation $\overline{C(\mcQ,N)}$ because the latter is the closure (in the product order topology) of ${C(\mcQ,N)}$.
\end{remark}

\begin{corollary}\label{cor:gkz:triangulation}
The closed cones $\overline{C(\mcQ,N)}$, where $\mcQ$ runs over all regular triangulations of $(P,A)$, cover $\MNaR$. These cones have all the same dimension, it is $N \cdot|A| = \dim\MNaR$.
\end{corollary}

\begin{proof}
Each polytopal subdivision can be refined into a regular triangulation. Since $\bfSigma_N$ is complete, this implies that given $\Psi\in \MNaR$, there exists a regular triangulation $\mcQ$ such that $\Psi\in \overline{C(\mcQ,N)}$. It follows that the closed cones associated to a regular triangulation cover $\MNaR$.
In \cite[Chapter 7, Corollary 2.6]{GKZ} it is proved that $\dim\overline{C(\mcQ,1)} = |A|$ for a regular triangulation $\mcQ$. The formula for an arbitrary $N$ follows from Corollary \ref{cor:gkz:dimension}.
\end{proof}

\begin{remark}\label{remark:generic:matrix}\rm
For $N=1$ we are back in the case studied by Gelfand, Kapranov and Zelevinsky, where the Euclidean topology coincides with the order topology. In particular, they show (and it is also a consequence of Corollary~\ref{cor:gkz:triangulation}), that a \emph{``generic element''} $\Psi\in M_{1,r}(\mathbb R)$ is an element in $C(\mcQ,1)$ for a regular triangulation.

This holds also for $N>1$ if one uses the term \emph{``generic element''} with respect to Euclidean topology.
Indeed, let $\mcQ$ be a polyhedral subdivision which is not a regular triangulation. By Proposition~\ref{prp:gkz:as:polar}, $\overline{C(\mcQ,N)}$ is a $\mu$-polyhedral cone in $M_{N,r}(\mathbb R)$. So its closure with respect to the Euclidean topology is by roposition~\ref{prp:polar:dimension} a polyhedral cone in  $M_{N,r}(\mathbb R)$, which is of dimension strictly less than $N\cdot |A|$. It follows that there exists an open and dense subset $U$ (in the Euclidean topology) in $M_{N,r}(\mathbb R)$, such that $\Psi\in M_{N,r}(\mathbb R)$ generic (i.e. $\Psi\in U$) implies $\Psi$ is an element in $C(\mcQ,N)$ for a regular triangulation.
\end{remark}

\begin{corollary}\label{cor:minimal:gkz_cone}
The polytopal subdivision $\mathcal P=(P,A)$ of $(P,A)$ is regular. The closed GKZ-cone $\overline{C(\mathcal P,N)}$ is the unique minimal cone in the fan $\bfSigma_N$. It is equal to the open GKZ-cone $C(\mathcal P, N)$, which is a vector subspace of $\MNaR$ of dimension $N\cdot n$.
\end{corollary}

\begin{proof}
The subdivision $\mathcal P= (P,A)$ is regular by definition since $\mathcal P$ does not refine any other subdivision. For the same reason $V = \overline{C(\mathcal P,N)}$ is the minimal cone in $\bfSigma_N$. In particular $C(\mathcal P, N) = V$ by Proposition \ref{prp:gkz:open:closed}.

The GKZ-condition cone $\mcC_{\mathcal P}$ has generators given by pairs of vectors $\pm\uu_{v,\bfDelta}$ since $A_P=A_{\mathcal P} = A$. This proves that $V$ is a vector subspace of $\MNaR$.

For its dimension, fix an affine basis $\bfDelta\subseteq A$ and note that a function $g_{\Psi,\mathcal P}$ with $\Psi\in V$ is completely determinated by its values on $\bfDelta$. Moreover such values can be  chosen freely. This proves that $\dim V = N\cdot n$.
\end{proof}

\section{Toric varieties and quasi-valuations arising from GKZ-cones}\label{sec:quasi-valuations}

\emph{Throughout the rest of the article we fix $\mathbb{K}$ to be an algebraically closed field of characteristic zero.}

In this section, we start to investigate an application of the higher rank GKZ-cones in the study of semi-toric degenerations (a flat degeneration into a finite union of toric varieties) of the toric variety associated to the polytope $P$ and an embedding arising from the markings in $A$. Some special examples of such degenerations have been studied by the authors from an algebro-geometric point of view under the framework of Seshadri stratifications in \cite{CCFL}.

\subsection{The toric variety $X_{A}$ and an embedding}\label{subsec:toric-variety}

Our standard reference for toric varieties is \cite{CLS}.

Let $T\simeq (\mathbb K^*)^{n-1}$ be a torus with character lattice $M\simeq \mathbb Z^{n-1}$ and dual lattice $N$. We write $\langle\cdot ,\cdot \rangle:N\times M\to\mathbb{Z}$
for the non-degenerate pairing defined by: $\langle\eta,\mu\rangle\in \mathbb Z$ is the unique integer such that $\mu(\eta(s))=s^{\langle\eta,\mu\rangle}$ for all $s\in\mathbb K^*$.

As in Section~\ref{subsec:polytopes:functions}, let $A=\{\chi_1,\ldots,\chi_r\}\subseteq M$ be a finite subset and let $P\subseteq M_\mathbb R:=M\otimes_{\mathbb Z}\mathbb R$ be the convex hull of $A$. We assume without loss of generality that $P$  is a full dimensional polytope. Let $\{e_\chi\mid \chi\in A\}\subseteq \mathbb K^A$  be the standard basis of $V:=\mathbb K^A$. 

\begin{definition}{\cite[Definition 2.1.1]{CLS}}\label{definition:toric:variety}
The \textit{embedded toric variety $X_{A}\subseteq \mathbb P(V)$} is defined as the Zariski closure of the image of the  map 
$\iota$:
\begin{equation}\label{embedding:toric}
\iota:T \rightarrow \mathbb P(V),\quad t \mapsto \left[\sum_{\chi\in A} \chi(t) e_\chi\right].
\end{equation}
\end{definition}

Note that the variety  $X_{A}$ is not necessarily normal, and it is a toric variety in the usual sense for a torus which might be a finite quotient of $T$. Denote by $\hat X_{A}\subseteq V$ the affine cone over  $X_{A}$. Let $\hat T$ be the torus $\mathbb K^*\times T$ with character lattice $\mathbb Z\times M$. Let  $\hat\iota:\hat T\rightarrow V$ be the map $(c,t)  \mapsto \sum_{\chi\in A} \chi(t) ce_\chi$. Then $\hat X_{A}$ is the closure of the image of $\hat\iota$.

We write $x_{\chi}$ for the linear function dual to $e_\chi$, $\chi\in A$. The ring of polynomial functions on $V=\mathbb K^A$ is identified with the polynomial ring $\mathbb{K}[x_{\chi_1},\ldots,x_{\chi_r}]$.

The homogeneous coordinate ring $\mathbb K[X_{A}]$ of the embedded variety $X_{A}\subseteq \mathbb P(V)$ and the coordinate ring $\mathbb K[\hat X_{A}]$ of the affine cone $\hat X_{A}$ are the same rings. Since we often work with $\hat X_A$, we stick to the notation $\mathbb K[\hat X_A]$. The variety $\hat X_{A}\subseteq V$ is irreducible. It follows that $\mathbb K[\hat X_A]$ is an integral domain which is naturally graded: 
$$\mathbb K[\hat X_A]=\bigoplus_{m\ge 0}\mathbb K[\hat X_A]_m.$$

Let $K(P)\subseteq \mathbb R\oplus M_\mathbb R\simeq \R^n$ be the cone over the polytope $P$. To avoid double indexes later, we proceed as in Section~\ref{subsec:polytopes:functions}:
\par\vskip 5pt
\noindent{\bf Notation:}
We write $u,v,\ldots$ for elements in $\R^{n-1}$ and $\R^n$. Only the elements of $A$ are  denoted by greek letters $\chi$, both as elements of $\R^{n-1}$  and as elements of  $\R^n$, it will be clear from the context where they are.
\par\vskip 5pt
Let $S\subseteq K(P)$ be the monoid generated by $A$, we have $S\subseteq K(P)\cap (\mathbb Z\times M)$. Each $u\in S$ can be written as a non-negative integral linear combination $u= \alpha_1\chi_1+\ldots +\alpha_r\chi_r$ of elements in $A$. The monomial $x_{\chi_1}^{\alpha_1}\cdots x_{\chi_r}^{\alpha_r}$ in $\mathbb K[V]$ is a homogeneous $\hat T$-eigenvector of degree $d_u:=\alpha_1+\ldots+\alpha_r$. Its restriction $f_u=x_{\chi_1}^{a_1}\cdots x_{\chi_r}^{a_r}\vert_{\hat X_{A}}\in \mathbb K[\hat X_A]$ is non-zero; it is a $\hat T$-eigenvector of degree $d_u$ and associated to the $T$-character $\eta_u=\alpha_1\chi_1+\ldots+\alpha_r\chi_r$. Only if we want to refer to the degree and the character associated to $f_u$ we write $u=\alpha_1(1,\chi_1)+\ldots+\alpha_r(1,\chi_r)=(d_u,\eta_u)$.

The following is well known, see \cite{CLS}.
\begin{lemma}\label{explicit:basis:two:two}
\begin{itemize}
\item[{\rm (i)}] The set $\{f_{u}\mid u\in S \}$ forms a basis for $\mathbb K[\hat X_A]$. The functions $f_u$
depend only on $u$ and not on the decomposition $u= \alpha_1\chi_1+\ldots +\alpha_r\chi_r$ of $u$.
\item[{\rm (ii)}] The following functions linearly span the vanishing ideal $I_A\subseteq \mathbb K[V]$ of  $ X_{A}\subseteq \mathbb P(V)$:
$$
\left\{x_{\chi_1}^{\beta_1}\cdots x_{\chi_r}^{\beta_r} -x_{\chi_1}^{\alpha_1}\cdots x_{\chi_r}^{\alpha_r}\,\left\vert\, 
\begin{array}{c}
u \in S, \quad x_{\chi_1}^{\beta_1}\cdots x_{\chi_r}^{\beta_r} \not=x_{\chi_1}^{\alpha_1}\cdots x_{\chi_r}^{\alpha_r}\\
\sum_{i=1}^r \beta_i \chi_i=(d_u,\eta_u)=\sum_{i=1}^r \alpha_i \chi_i. \\
\end{array}\right.\right\}
$$
\end{itemize}
\end{lemma}

%
%

\subsection{Generalities on quasi-valuations and examples}\label{some:generalities:quasi:valuation}

For $N\ge 1$ consider the additive group $(\mathbb Q^N,+)$, endowed with the lexicographic order. The order is compatible with the group structure of $(\mathbb Q^N,+)$ and multiplication by positive rational numbers. To simplify the notation, we often tacitly add an element $\infty$ to the ordered group $\mathbb{Q}^N$ which is assumed to be larger than any other element, and any element, when added to $\infty$, equals $\infty$.

Let $\mathfrak R$ be a finitely generated integral reduced algebra.

\begin{definition}\label{defn:quasi:valuation}
A \emph{quasi-valuation} on $\mathfrak R$ with values in the ordered group $(\mathbb Q^N,+)$ is a map $\nu: \mathfrak R\to \mathbb Q^N\cup\{\infty\}$ having the following properties for any $x,y\in \mathfrak R$ and $\lambda\in\mathbb{K}^*$:
\begin{enumerate}
\item[(0)] $\nu(x) = \infty$ if and only if $x=0$;
\item[(1)] $\nu(x + y) \ge \min\{\nu(x), \nu(y)\}$ (minimum property);
\item[(2)]  $\nu(\lambda x) = \nu(x)$;
\item[(3)]  $\nu(xy) \ge \nu(x) + \nu(y)$ (super-additivity).
\end{enumerate}
\end{definition}

Property (3) implies $\nu(x^m)\ge m\nu(x)$ for $m\in \mathbb N$. The quasi-valuation $\nu$ is called \emph{radical} if $\nu(x^m)=m\nu(x)$ for all $m\in\mathbb N$.  If $\nu$ is always additive, i.e.  $\nu(xy)= \nu(x) + \nu(y)$ for all $x,y \in \mathfrak R$, then it is called a \emph{valuation}.
\vskip 5pt
\textit{Throughout the following we assume the quasi-valuation to be \emph{discrete}, \textit{i.e.}, the image, without $\infty$, is a discrete subset of $\mathbb Q^N$. }
\vskip 5pt
\begin{example}\label{example:homomorphism:valuation1}
Let $X_A\subseteq \mathbb P(V)$ be a toric variety as in Section~\ref{subsec:toric-variety} and let
$\phi: \mathbb Q\oplus M_\mathbb Q\rightarrow \mathbb Q^N$ be a $\mathbb{Q}$-linear map. We use $\phi$ to define a valuation $\nu_{\phi}:\mathbb K[\hat X_A]\to\mathbb{Q}^N\cup\{\infty\}$. Indeed, we set $\nu_\phi(0)=\infty$, and for $f=\sum_{u\in S} a_{u}f_{u}\neq 0$ in $\mathbb K[\hat X_A]$ we define:
$$\nu_\phi(f):=\min\{\phi(u)\mid a_{u}\not=0\}.$$
It is now easy to check that $\nu_\phi$ is not affected by (non-zero) scalar multiplication, 
it is additive and has the minimum property, \emph{i.e.} it is a valuation.
\end{example}

\begin{example}\label{example:quasi:valuation1}
Let $X_A\subseteq \mathbb P(V)$ be a toric variety as in Section~\ref{subsec:toric-variety} and let $\phi_1,\ldots,\phi_p$ be $\mathbb{Q}$-linear maps $\phi_j:\mathbb Q\oplus M_\mathbb Q\rightarrow \mathbb Q^N$,  $j=1,\ldots,p$. Using Example~\ref{example:homomorphism:valuation1} we set:
$$\mathcal V_{\phi_1,\ldots,\phi_p}:\mathbb K[\hat X_{A}]\rightarrow \mathbb Q^N\cup\{\infty\},\quad f\mapsto \min\{\nu_{\phi_1}(f),\ldots,\nu_{\phi_p}(f)\}.$$
It is well known that the minimum over a finite collection of valuations is a quasi-valution (see, for example \cite[Proposition 4.1]{FL}). The quasi-valuation is radical by construction.
\end{example}

\begin{remark}\rm
There is no standard terminology in the literature for the names given to quasi-valuations and its special properties. Quasi-valuations are called \emph{pseudo-valuations} \cite{B}, and they are a special case of \emph{loose valuations} \cite{T}. The name \emph{homogeneous} is sometimes used as a synonym for what is called \emph{radical} above.
\end{remark}

Quasi-valuations are closely related to filtrations on algebras. Given a quasi-valuation $\nu:\mathfrak R\rightarrow \mathbb Q^N\cup\{\infty\}$ on the $\mathbb K$-algebra $\mathfrak R$, we can define the associated filtration $\mathcal F_\nu=\{\mathcal F_{\ge \mathbf{a}}\}_{\mathbf{a}\in\mathbb Q^N}$ by subspaces:
$$\mathcal F_{\ge \mathbf a}= \{f\in \mathfrak R_{}\mid \nu(f)\ge \mathbf{a}\}\ \ \text{and}\ \ 
\mathcal F_{> \mathbf a}= \{f\in \mathfrak R_{}\mid \nu(f)> \mathbf{a}\}.$$
Since $\nu$ is a quasi-valuation, we have $\mathcal F_{\ge \mathbf a}\cdot \mathcal F_{\ge \mathbf a'}\subseteq \mathcal F_{\ge \mathbf a+\mathbf a'}$. The subquotients $\mathcal F_{\ge \mathbf a}/\mathcal F_{> \mathbf a}$, $\mathbf{a}\in\mathbb Q^N$, are called \emph{leaves} of the filtration. 

\begin{definition}\label{definition:associated:graded:algebra}\rm
The \emph{associated graded algebra} is denoted $\mathrm{gr}_\nu\mathfrak R
=\bigoplus_{\mathbf a\in\mathbb{Q}^N}\mathcal F_{\ge\mathbf a}/\mathcal F_{>\mathbf a}$.
\end{definition}

\subsection{Quasi-valuations from $\mathcal{Q}$-piecewise-linear maps}\label{quasi-valuation:one}

Throughout the following sections, in most of the cases, only \emph{rational} matrices will be considered. We keep the notation as in Section~\ref{subsec:toric-variety}, so $(P,A)$ is a marked polytope with $A=\{\chi_1,\ldots,\chi_r\}$. 

For a matrix $\Psi\in M_{N,r}(\mathbb Q)$ let $\mathcal Q=(Q_i,A_i)_{i\in I}$ be the regular polytopal subdivision of $(P,A)$ such that $\Psi\in C(\mathcal Q,N)$ (such a subdivision exists by Theorem~\ref{thm:gkz:complete}). Let $g_{\Psi}:K(P)\rightarrow \mathbb R^N$ be the unique associated piecewise-linear map defined on $K(P)$ (see Theorem~\ref{thm:gkz:complete}, and Corollary~\ref{cor:open:partition}
for the existence and uniqueness, see \eqref{definition:gPsi} for the definition).

By Proposition~\ref{prop:convex:up:=:minimum}, $g_{\Psi}$ is the minimum
over a finite number of linear maps $\psi_Q:\mathbb R^n\rightarrow \mathbb R^N$, where $Q$ is running in the set of polytopes in the polytopal subdivision $\mathcal Q$. By assumption, $\Psi$ has entries in $\mathbb{Q}$, so by construction
we get for every marked polytope $(Q,A_Q)$ in the polytopal subdivision $\mathcal Q$ an induced linear map $\psi_Q:\mathbb Q^n=\mathbb Q\oplus M_\mathbb Q\rightarrow \mathbb Q^N$ which is uniquely determined by the condition: if $\chi_j\in A_Q$, $1\le j\le r$, then $\psi_Q(\chi_j)=\Psi_j$, the $j$-th column of $\Psi$.

The semigroup $S$ is an indexing system for the  basis $\{f_{u}\mid u\in S \}$ of the homogeneous coordinate ring $\mathbb K[\hat X_A]$ (see Lemma~\ref{explicit:basis:two:two}). By construction we have $S\subseteq K(P)$. It follows from Example~\ref{example:quasi:valuation1}:

\begin{proposition}\label{proposition:quasi:valuation}
The map $\mathcal V_\Psi: \mathbb K[\hat X_A]\rightarrow \mathbb Q^N\cup\{\infty\}$ defined  for $f=\sum_{u\in S} a_{u}f_{u}\neq 0$ by  
\begin{equation}\label{quasi-one}
\mathcal V_\Psi(f):= \min\{\psi_Q(u)\mid Q\in\mathcal Q,  a_{u}\not=0\}= \min\{g_{\Psi}(u)\mid  a_{u}\not=0\},
\end{equation}
is a radical quasi-valuation.
\end{proposition}

The convex upward condition on $g_{\Psi}$ (see \eqref{definition:convex}) turns into the super-additivity property for $\mathcal V_\Psi$. The minimum property of the quasi-valuation can be made more precise. For $f=\sum_{u\in S} a_{u}f_{u}\neq 0$ denote by $\Delta_f\subseteq \mathbb R\oplus M_\mathbb R$ the polytope obtained as the convex hull of $\{u\mid a_{u}\not=0\}$. Let $V_f$ be the set of vertices of $\Delta_f$.

\begin{lemma}\label{lemma:extremal:value}
For $f=\sum_{u\in S} a_u f_{u}\in  \mathbb K[\hat X_A]\setminus\{0\}$ one has
$$\mathcal V_\Psi(f)= \min\{g_\Psi(u)\mid u\in \Delta_f\}= \min\{\mathcal V_\Psi(f_{u})\mid u\in V_f\}.$$
\end{lemma}

\begin{proof}
If $a_u\not=0$ and $u$ is not a vertex of $\Delta_f$, then write $u=\sum_{w\in V_f} b_{w} w$ as a convex linear combination of the vertices.  
The convex upward property \eqref{definition:convex} implies:
$$\begin{array}{rcl}
g_{\Psi}(u)=g_{\Psi}(\sum_{w\in V_f} b_{w} w)
\ge \sum_{w\in V_f} b_{w} g_{\Psi}(w)
&\ge&\min\{g_{\Psi}(w)\mid w\in V_f\}\\
&=&\min\{\mathcal V_\Psi(f_{w})\mid w\in V_f\},
\end{array}$$
which implies the claim.
\end{proof}

\section{Semi-toric degenerations from GKZ-cones}\label{two:graded:algebras}

Given $\Psi\in M_{N,r}(\mathbb Q)$, let $\mathcal Q$ be such that $\Psi\in C(\mathcal Q,N)$, we have the associated graded algebra $\mathrm{gr}_{\mathcal V_\Psi} \mathbb K[\hat X_A]$. By Lemma~\ref{lemma:extremal:value}, it has the classes $\{\bar f_u\mid u\in S\}$ as a vector space basis. For a polytope $Q\in\mathcal Q$ let $S_Q=S\cap K(Q)$ be the intersection of $S$ with the cone $K(Q)$ over $Q$. The goal of this section is to study the geometry of these associated graded algebras.

\begin{theorem}\label{coro:geometric1}
The associated graded algebra $\mathrm{gr}_{\mathcal V_\Psi}\mathbb K[\hat X_A]$ is finitely generated. The variety $X_0=\mathrm{Proj} (\mathrm{gr}_{\mathcal V_\Psi}\mathbb K[\hat X_A])$ is reduced, it is the irredundant union of the toric varieties $X_Q=\mathrm{Proj}(\mathbb K[S_Q])$, where $Q$ runs over the set of all polytopes  in $\mathcal Q$. The variety $X_0$ is equidimensional of dimension $\dim X_A$. In particular, it depends only on $\mathcal Q$ and not on the choice of $\Psi\in C(\mathcal Q,N)\cap M_{N,r}(\mathbb{Q})$.
\end{theorem}

\begin{proof}
It suffices to give the multiplication rule for the classes $\bar f_u$, $u\in S$. Since $g_\Psi$ is convex upward and the polytopes in $\mathcal Q$ are exactly the domains of linearity, one gets: $\bar f_{u}\bar f_{w}= \bar f_{u+w}$ if there is a polytope $ Q\in \mathcal Q$ such that $u,w\in S_Q$, and $\bar f_{u}\bar f_{w}=0$ otherwise. In particular, the algebra is finitely generated and has no nilpotent elements, so $X_0 = \textrm{Proj}(\mathrm{gr}_{\mathcal V_\Psi} \mathbb K[\hat X_A])$ is reduced. 

For $Q\in\mathcal Q$ set  $\bar f_Q :=\prod_{u} \bar f_u\in  \mathrm{gr}_{\mathcal V_\Psi} \mathbb K[\hat X_A]$, where the product is running over all elements $u\in S_Q$ such that $\deg f_u=1$. 
Let $I(Q)$ be the annihilator of $\bar f_Q$ in $\mathrm{gr}_{\mathcal V_\Psi} \mathbb K[\hat X_A]$. By the product rule, $\bar f_w \not\in I(Q)$ if and only if $w\in S_Q$, and hence the intersection $\bigcap_{Q\in\mathcal Q} I(Q)= (0)$ is the zero ideal. Moreover, the quotient $\mathrm{gr}_{\mathcal V_\Psi} \mathbb K[\hat X_A]/I(Q)$  is isomorphic to $\mathbb K[S_Q]$, an algebra without zero-divisors. Hence $I(Q)$ is a prime ideal. To show that $I(Q)$ is a minimal prime ideal, suppose $I\subsetneq I(Q)$ is an ideal. Then there exists an element $u\in  S$, $u\not\in S_Q$, such that $\bar f_u\not\equiv 0\pmod I$. But $\bar f_Q \bar f_u = 0$, so the quotient $\mathrm{gr}_{\mathcal V_\Psi} \mathbb K[\hat X_A]/I$ has a non-trivial zero divisor and $I$ is hence not a prime ideal. It follows that $\bigcap_Q I(Q)= (0)$ is the minimal prime decomposition of the zero ideal in $\mathrm{gr}_{\mathcal V_\Psi} \mathbb K[\hat X_A]$. For a polytope $Q'\in \mathcal Q$, $\bar f_{Q'}$ is a non-zero element in the intersection $\bigcap_{Q\not=Q'} I(Q)$. This shows that the intersection $\bigcap_Q I(Q)$, running over all $Q \in\mathcal  Q$, is non-redundant. An irreducible component of $X_0$ is hence isomorphic to $X_Q = \textrm{Proj\,} (\mathbb K[S_Q])$ for some polytope $Q\in\mathcal Q$, it follows $\dim X_Q = \dim X_A$.
\end{proof}

\begin{remark}
Assume that the quasi-valuation $\mathcal{V}_\Psi$ is of full rank. For a fixed polytope $Q\in\mathcal{Q}$, the restriction of $\mathcal{V}_\Psi$ to the subalgebra of $\mathbb{K}[\hat{X}_A]$ generated by $f_u$ for $u\in S_Q$ is a valuation. Therefore the image $\Gamma:=\mathrm{im}(\mathcal{V}_\Psi)$ is a fan of monoids. Let $\mathbb{K}[\Gamma]$ be the associated fan algebra (see \cite{CFL1} for this notion). 

Let $u,u'\in S$. It follows from the piecewise-linear property of $g_{\Psi,\mathcal{Q}}$ that if there exists a polytope $Q\in\mathcal{Q}$ such that $u,u'\in S_Q$, then $\overline{f}_u\cdot \overline{f}_{u'}=\overline{f}_{u+u'}$ in the associated graded algebra $\mathrm{gr}_{\mathcal{V}_\Psi}\mathbb{K}[\hat{X}_A]$; otherwise the product of $\overline{f}_u$ and $\overline{f}_{u'}$ is zero. The property of $\mathcal{V}_\Psi$ in the above paragraph and the full rank assumption imply that the relations in $\mathbb{K}[\Gamma]$ are exactly the same as in the associated graded algbera. It then follows that the fan algebra $\mathbb{K}[\Gamma]$ is isomorphic to the associated graded algebra $\mathrm{gr}_{\mathcal{V}_\Psi}\mathbb{K}[\hat{X}_A]$.

Similar arguments as in \cite[Proposition 3]{And} can be applied to show that there exists a flat family $\pi:\mathcal{X}\to\mathbb{A}^1$ such that the special fiber $\pi^{-1}(0)$ is isomorphic to $\mathrm{Proj}(\mathbb{K}[\Gamma])$ (hence it is isomorphic to $\mathrm{Proj}(\mathrm{gr}_{\mathcal{V}_\Psi}\mathbb{K}[\hat{X}_A])$), and the generic fibers are isomorphic to $X_A$. A semi-toric variety is a union of toric varieties, and the flat degeneration $\pi^{-1}(0)$ is called a semi-toric degeneration of $X_A$.

This argument is only sketched in this paper for completeness, the Rees algebra and its higher rank generalizations will be studied systematically in a separate work \cite{CCFL3}.
\end{remark}

We restate Theorem \ref{coro:geometric1} in the language of Khovanskii basis for a quasi-valuation \cite[Section 15.5]{CFL1}.
A subset $\mathbb{B}\subseteq\mathbb{K}[\hat X_A]\setminus\{0\}$ is called a Khovanskii basis for a quasi-valuation $\nu:\mathbb{K}[\hat X_A]\to\mathbb{Q}^N\cup\{\infty\}$, if the image of $\mathbb{B}$ in $\mathrm{gr}_{\nu}\mathbb{K}[\hat X_A]$ generates the associated graded algebra $\mathrm{gr}_{\nu}\mathbb{K}[\hat X_A]$.

\begin{corollary}
For any regular subdivision $\mathcal{Q}$ and $\Psi\in C(\mathcal{Q},N)$, there exists a finite Khovanskii basis for $\mathcal{V}_{\Psi}$.
\end{corollary}

\appendix
\section{Appendix: Some remarks on the order topology}\label{order:topology}

In the following we consider $\mathbb{R}^N$ as the topological space endowed with the order topology.
More precisely, for $\ba=(a_N,\ldots,a_1)$ and $\bb=(b_N,\ldots,b_1)\in\mathbb{R}^N$, $\ba<\bb$ 
with respect to the lexicographic order, consider the interval:
$$
(\ba,\bb):=\{\bc\in\mathbb{R}^N\mid \ba<\bc<\bb\}.
$$
The order topology on $\mathbb{R}^N$ has a basis consisting of the intervals $(\ba,\bb)$.

The following results on the order topology are not difficult, the proofs are left to the readers.

\begin{lemma}\label{Lem:TopGroup}
The additive group $(\mathbb R^N,+)$, endowed with the order topology, is a topological group. 
\end{lemma}

\begin{lemma}
Given a fixed real number $\lambda$, the multiplication map $\mu_\lambda:\mathbb{R}^N\rightarrow \mathbb{R}^N$, $\ba\mapsto \lambda\ba$, is continuous in the order topology.
\end{lemma}

\begin{lemma}\label{finer_topology}
The order topology on $\mathbb{R}^N$ coincides with the Euclidean topology only for $N=1$, for
$N\ge 2$, the order topology is strictly finer.
\end{lemma}

We endow the space $M_{N,r}(\mathbb R):=\mathbb{R}^N\times\ldots\times \mathbb{R}^N$ with 
the product of the order topology.  A basis of the topology consists of the product of intervals $(\ba_1,\bb_1)\times \cdots\times (\ba_r,\bb_r)$ with $\ba_1<\bb_1$, $\ldots$, $\ba_r<\bb_r$. We call the topology the \emph{product order topology} on $M_{N,r}(\mathbb R)$. 

\begin{lemma}\label{lem:continuous}
Let $\uv=(v_r,\ldots,v_1)^T\in\mathbb R^r$ be a fixed vector. The natural map $\mu_\uv:M_{N,r}(\mathbb R)\rightarrow \mathbb{R}^N$, $\Psi\mapsto \Psi\cdot\uv$, is continuous with respect to the product order topology on $M_{N,r}(\mathbb R)$ and the order topology on $\mathbb R^N$.
\end{lemma}

\begin{lemma}\label{lemma:cone:closed:open}
The closed half-space 
$H_{{\uv}}^\leq =\{\Psi\in M_{N,r}(\mathbb R)\mid \Psi\cdot\uv\le \mathbf 0\}$ is closed in the product order topology, and the open half-space $H_{ {\uv}}^< =\{\Psi\in M_{N,r}(\mathbb R)\mid \Psi\cdot\uv < \mathbf 0\}$ is open in the product order topology. The subspace $H_{ {\uv}}^0 =\{\Psi\in M_{N,r}(\mathbb R)\mid \Psi\cdot\uv = \mathbf 0\}$ is a closed subspace.
 \end{lemma}

\end{document}